\documentclass[a4paper]{article}
\usepackage{hyperref}
\usepackage{amsmath}
\usepackage{amsfonts}
\usepackage{amsthm}
\usepackage{indentfirst}
\usepackage{geometry}
\usepackage{color}

\renewcommand{\theequation}{\thesection.\arabic{equation}}
\usepackage{indentfirst}
\date{}
\makeatletter
\@addtoreset{equation}{section}

\newtheorem{theorem}{Theorem}[section]
\newtheorem{lemma}[theorem]{Lemma}

\newtheorem{rem}{Remark}[section]

\begin{document}

\title{Distribution solutions of a static dispersion Schr\"{o}dinger equation}

\author{Tiantian Zhou and Yutian Lei
}

\date{}
\maketitle


\begin{abstract}
In this paper, we study qualitative properties of distribution solutions of a fourth order equation
$$
-\Delta u(x)+a^2\Delta^2u(x)=u^q(x), \quad u(x)>0 \ \ in \ \ \mathbb{R}^3,
$$
where $a>0$ and $q>0$. It is the static equation of a mixed dispersion Schrodinger equation,
and also the Euler-Lagrange equation satisfied by extremal functions of an embedding inequality.
We obtain some Liouville theorems and the corresponding related critical exponents,
which imply the best constant of the embedding inequality cannot be attainable.
We also obtain some regularity results (involving differentiability, integrability,
radial symmetry) and asymptotics at infinity of distribution solutions.
Here an equivalent integral equation with the Coulomb potential $|x|^{-1}(1-e^{-|x|/a})$ plays a key role.
In addition, we also use the Pohozaev identity in integral form to obtain the Liouville theorem of this integral equation.
Such the Pohozaev identity still works to handle the Allen-Cahn-type integral equation.
\end{abstract}

\noindent{\bf{Keywords}}: dispersion Schr\"{o}dinger equation,
distribution solutions, integral equation, Coulomb potential
\par
\noindent{\bf{MSC2020}}: 35Q55, 35Q60, 35J91, 45E10

\newtheorem{proposition}[theorem]{Proposition}

\renewcommand{\theequation}{\thesection.\arabic{equation}}
\catcode`@=11
\@addtoreset{equation}{section}
\catcode`@=12

\section{Introduction}

The following mixed dispersion Schr\"odinger equations
\begin{equation}\label{sf1}
    i\psi_t+\alpha \Delta\psi+\beta\Delta^2{\psi}
    +F(|\psi|^2)\psi=0, \quad x \in \mathbb{R}^3, \quad t \geq 0,
\end{equation}
have been introduced by Karpman and
Shagalov \cite{Kman,KS} to take into account the role of small
fourth-order dispersion terms in the propagation of
intense laser beams in a bulk medium with Kerr nonlinearity.
Here $\alpha,\beta \in \mathbb{R}$, and $F(t)$ is a power-like nonlinear function.
The importance of \eqref{sf1} in
mathematical physics is a source of many interesting mathematical
problems (cf. \cite{BKS,BC,BCGJ,BN,BL,FI,PX} and many others).
To study standing wave solutions,
inserting $\psi(x,t)=e^{i\gamma t}u(x)$ into \eqref{sf1},
we have the static four order elliptic equation
\begin{equation}\label{sf2}
    \alpha \Delta u+\beta\Delta^2 u+F(|u|^2)u-\gamma u=0.
\end{equation}
The existence and asymptotic behavior of solutions of \eqref{sf2}
are helpful to well understand the global existence and scattering
of solutions of evolution equation \eqref{sf1} (cf. \cite{BC,FI,Pau}
and the references therein). Eq. \eqref{sf2} appears also in the
study of fourth-order hyperbolic equations \cite{Lsky,Psd}.
In addition, the fourth
order elliptic operator $\alpha \Delta+\beta\Delta^2$ also appears in the
model of electrostatic actuation (cf. \cite{LW,LY})
and the Bopp-Podolsky electrostatic theory (cf. \cite{Bopp,PS,Gao,Pd,QSS,SS}).

Although the nonlinear term $F(t)$
in \cite{LS,LSS} is not the power function,
Remark 3 in \cite{PX} shows that the simple form of the power-like
$F(t)$ is convenient. In this paper, we
assume $F(t)=t^{\sigma}+d_0$ with $\sigma>0$ and $d_0 \in \mathbb{R}$.
When $d_0=0$, the following functional
\begin{equation}\label{wei}
E(u)=\frac{\alpha}{2}\int_{\mathbb{R}^3}|\nabla u|^2dx
+\frac{\beta}{2}\int_{\mathbb{R}^3}|\Delta u|^2dx
-\frac{1}{2(\sigma+1)}\int_{\mathbb{R}^3}|u|^{2(\sigma+1)}dx.
\end{equation}
plays a key role to study the stability of standing solutions of \eqref{sf1}
(cf. \cite{BC,BL,Pau}). The minimizers of \eqref{wei}
with the mass constriction are the ground states of \eqref{sf2}.

\subsection{Problems}

In this paper, we are concerned with the existence and asymptotic behavior of solutions
of \eqref{sf2} in the cases of $d_0=\gamma-1$ and $d_0=\gamma$.

\vskip 3mm

{\bf I. Case of $d_0=\gamma-1$.}

When $d_0=\gamma-1$, \eqref{sf2} becomes (cf. (20) in \cite{Kman} and (1.4) in \cite{NP})
\begin{equation}\label{sf3}
    \alpha \Delta u+\beta\Delta^2{u}+(|u|^{2\sigma}-1)u=0.
\end{equation}
The asymptotic behavior of $u$ at infinity plays an important role in study of the stability of the soliton solutions
and implies a necessary conditions of the stabilization of self-focusing and collapse by high-order dispersion.
Two natural problems are when $\lim_{|x| \to \infty}|u| \in \{0,1\}$ and what the asymptotic rates are.

Eq. \eqref{sf3} is also an Allen-Cahn model in phase transition theory, where
$u$ represents the order parameter. Namely, the domains where $|u| \in \{0,1\}$ presents two phase,
respectively. The Liouville theorem of \eqref{sf3}
shows the distribution of phases in a mixed state. Therefore, a natural problem
is when \eqref{sf3} has only the trivial solutions.

When $\alpha=-1$, $\beta=0$ and $\sigma=1$, \eqref{sf3} becomes the Ginzburg-Landau equation
$$
-\Delta u=u(1-|u|^2).
$$
In 1994, Shafrir estimated the asymptotic rates of $1-|u|^2$
when $u$ is a finite penalization term solution (cf. \cite{Shafrir}).
Afterwards, the Liouville theorems for the finite energy solution and
the finite penalization term solution can be found in \cite{BMR} and
\cite{Frn} respectively.

\vskip 3mm

{\bf II. Case of $d_0=\gamma$.}

When $d_0=\gamma$, \eqref{sf2} is the Euler-Lagrange equation satisfied
by the critical points of \eqref{wei} without any constriction. In addition,
the roles of (5) in \cite{MXZ} and (1.3) in \cite{Guo} imply that \eqref{sf2}
with $d_0=\gamma$ may have a hold on the scattering of global solutions and
the stability of standing wave solutions of \eqref{sf1}. This is usually embodied
by the existence and asymptotic behavior at infinity of solutions of the static equation
\eqref{sf2}.

On the other hand, \eqref{sf2} with $d_0=\gamma$ comes into play in the study of some
embedding inequalities.
Let $\alpha=1$, $\beta=-a^2$ and $2\sigma=q-1$. Then \eqref{sf2} becomes
\begin{equation}\label{A2}
    -\Delta u+a^2{\Delta}^2u=u^q, \quad u>0 \;\;\; in \;\; \mathbb{R}^3,
\end{equation}
where $a>0$, $q>0$. In 2019, d'Avenia and Siciliano introduce the space
$$
\mathcal{D}=\left\lbrace u\in D^{1,2}(\mathbb{R}^3):
\Delta u \in L^2(\mathbb{R}^3)\right\rbrace
$$
as the completion of $C_0^{\infty}(\mathbb{R}^3)$ with the respect to the
norm (Lemma 3.2 in \cite{PS})
$$
\|\phi\|_{\mathcal{D}}:=\sqrt{\left \| \nabla \phi  \right \|_{L^2(\mathbb{R}^3)}^2
+a^2\left \| \Delta \phi \right \|_{L^2(\mathbb{R}^3)}^2 }.
$$
Here $D^{1,2}(\mathbb{R}^3)$ is the homogeneous Sobolev space. In
particular, $\mathcal{D}$ is an Hilbert space continuously
embedded into $L^6(\mathbb{R}^3)$ and $L^{\infty}(\mathbb{R}^3)$ (Lemma 3.1 in
\cite{PS}). Therefore, for all $q \in [5,\infty]$,
\begin{equation}\label{zhao}
\left\| u\right\|_{L^{q+1}(\mathbb{R}^3)}
\leq C (\left \| \nabla u  \right \|_{L^2(\mathbb{R}^3)}
+a\left \| \Delta u \right \|_{L^2(\mathbb{R}^3)}),
\quad \forall \ u \in \mathcal{D}.
\end{equation}
Clearly, \eqref{A2} is also the Euler-Lagrange
equation satisfied by extremal functions of \eqref{zhao}.

In order to better understand \eqref{A2}, we introduce an integral equation
\begin{equation}\label{A1}
    u(x)=\int_{\mathbb{R}^3}\mathcal{K}_a(x-y)u^q(y)dy,
    \quad u>0 \ \ in \ \ \mathbb{R}^3,
\end{equation}
where $q>0$, $a>0$, and
$$
\mathcal{K}_a(x)=\frac{1-e^{-\frac{|x|}{a}}}{|x|}
$$
is the kernel function which is called the modified Coulomb energy
(or Coulomb potential) in \cite{Ld,LT} and \cite{Pd}.
In addition, $\mathcal{K}_a$ solves in distribution sense
\begin{equation}\label{peng}
-\Delta \phi+a^2{\Delta}^2\phi=4\pi \delta_0,
\end{equation}
and its energy is finite, i.e.,
$\|\mathcal{K}_a\|_{D} <\infty$ (cf. \cite{Bog,PS}).
In addition, $\Delta \mathcal{K}_a$ is the Yukawa potential, which
is the fundamental solution of the three-dimensional static Klein-Gordon equation
(cf. \S3.4 in \cite{MM}) and the modified Helmholtz equation (cf. (2)-(4) in \cite{AC}
and (2.7) in \cite{FK}).

If replacing the Coulomb potential by the Riesz potential in \eqref{A1},
we have
\begin{equation}\label{reisz}
    u(x)=\int_{\mathbb{R}^3} \frac{u^q(y)dy}{|x-y|^{n-\alpha}},
    \quad u>0 \ \ in \ \ \mathbb{R}^3.
\end{equation}
In critical case,
(\ref{reisz}) is related to the study of the best constant of a
simplified version of the Hardy-Littlewood-Sobolev inequality
(cf. (1) in \cite{CCL})
\begin{equation}\label{hls}
    \int_{\mathbb{R}^n}\int_{\mathbb{R}^n}\frac{f(x)f(y)}{|x-y|^{n-\alpha}}dxdy
    \leq C \left\| f\right\|_{L^{2n/(n+\alpha)}(\mathbb{R}^n)}^2,
    \quad \forall f \in L^{2n/(n+\alpha)}(\mathbb{R}^n).
\end{equation}
Here $0<\alpha<n$. The best constant of \eqref{hls} can be used to estimate
the upper bound of the Coulomb energy appearing in the Thomas-Fermi model
(cf. \cite{L94}). From \eqref{hls}, we can see easily that
\begin{equation}\label{xin}
\int_{\mathbb{R}^3}\int_{\mathbb{R}^3}\mathcal{K}_a(x-y)f(x)f(y)dxdy
\leq C\left\| f\right\|_{L^{6/5}(\mathbb{R}^3)}^2, \quad
\forall f \in L^{6/5}(\mathbb{R}^3).
\end{equation}
Clearly, (\ref{A1}) is the
Euler-Lagrange equation satisfied by extremal functions of
\eqref{xin}, where $u=f^{1/5}$ and $q=5$.
Unlike inequality \eqref{hls}, (\ref{xin}) has not rescaling invariance,
and hence $6/5$ is not the Sobolev-type critical exponent. In addition,
(1) in \cite{CCL} shows that \eqref{hls} has a clear extremal function
(see also \cite{CLO,Li} and \cite{Lieb}), but the best constant
of \eqref{xin} cannot be achievable (cf. Remark \ref{rem1.2}).

\subsection{Main results}

A positive function $u \in \mathcal{D}$ is called a distribution
solution of (\ref{A2}) if the following equality holds
\begin{equation}\label{weak}
    -\int_{\mathbb{R}^3} u \Delta \phi dx+a^2\int_{\mathbb{R}^3} u \Delta^2 \phi dx
    =\int_{\mathbb{R}^3}u^q\phi dx,
    \quad \forall \ \ \phi \in C_0^{\infty}(\mathbb{R}^3).
\end{equation}
Similarly, $u \in \mathcal{D}$ is called a positive distribution super
solution of (\ref{A2}) if for any nonnegative smooth function $\phi
\in C_0^{\infty}(\mathbb{R}^3)$, there holds
\begin{equation*}
    -\int_{\mathbb{R}^3} u \Delta \phi dx
    +a^2\int_{\mathbb{R}^3} u \Delta^2 \phi dx
    \geq \int_{\mathbb{R}^3}u^q\phi dx.
\end{equation*}

In this paper, we will investigate the existence, the regularity
(involving differentiability, integrability,
radial symmetry) and decay rates (when $|x| \to \infty$)
of distribution solutions of \eqref{A2}.

First, we consider the relation between integral equation (\ref{A1})
and PDE (\ref{A2}). If $u$ is a rapidly decreasing function, the
convolution property of the Dirac function implies that (\ref{A1})
is a distribution solution of (\ref{A2}). A natural problem is
whether (\ref{A1}) is equivalent to (\ref{A2}) when $u$ has some
integrability. Here we state the following equivalence result.

\begin{theorem}\label{theq} (Equivalence of \eqref{A1} and \eqref{A2})

    (i) If $u \in \mathcal{D}$ is a positive distribution solution of (\ref{A2}),
    then $u$ belongs to $L^{3(q-1)/2}(\mathbb{R}^3)$, and satisfies integral equation \eqref{A1} a.e. in
    $\mathbb{R}^3$.

    (ii) If $u \in L^{3(q-1)/2}(\mathbb{R}^3)$ is a positive solution of \eqref{A1},
    then $u \in \mathcal{D}$ is a distribution solution of (\ref{A2}).
\end{theorem}

\vskip 5mm

Next, we have the existence results of positive solutions of
(\ref{A2}) and \eqref{A1} (cf. Theorems \ref{th*1} and \ref{th1}).

\begin{theorem}\label{th*1} (Liouville theorem of \eqref{A2})

    (i) If (\ref{A2}) has a positive distribution super solution
    $u \in \mathcal{D}$, then $q>3$.

    (ii) If (\ref{A2}) has a positive distribution solution
    $u \in \mathcal{D}$, then $q>5$.
\end{theorem}

\begin{rem}
Compared with conformal equations,
the conclusions in Theorem \ref{th*1} are consistent with the
existence results of the Lane-Emden equation $-\Delta u=u^q$
(cf. \cite{GS}). Now, critical exponents $3$ and $5$ are the
Serrin exponent and the Sobolev exponent respectively.
However, these conclusions are different from
the existence result of $a^2\Delta^2 u=u^q$. In fact, \cite{Xu}
shows that $a^2\Delta^2 u=u^q$ has $C^4$-solution if and only
if $q=-7$. These results show that the role of operator $-\Delta$
is more prominent than that of operator $a^2\Delta^2$ in \eqref{A2}.
\end{rem}

Theorem \ref{th*1} (i) is the corollary of Theorem \ref{theq} (i)
and the following result (cf. Remark \ref{rem2.1}).

\begin{theorem}\label{th1} (Liouville theorem of \eqref{A1})

     (i) Eq. (\ref{A1}) has a positive super solution $u$
     in $L_{\rm loc}^\infty(\mathbb{R}^3)$, if and only if $q>3$.

     (ii) If (\ref{A1}) has a positive differentiable
     solution $u \in L^{q+1}(\mathbb{R}^3)$, then $q>5$.
\end{theorem}

\begin{rem}\label{rem1.2}
Theorem \ref{th*1} (ii) shows that $q>5$ is a necessary condition of the
existence of positive distribution solutions of (\ref{A2}) in $\mathcal{D}$.
Therefore, the best constant of \eqref{zhao} cannot be achievable
when $q=5$. Similarly, Theorem \ref{th1} (ii) shows that the best
constants of \eqref{xin} cannot be achievable (now, $q=5$).
\end{rem}

In the proof of Theorem \ref{th1} (i), we use the ideas of iteration (cf.
\cite{cdm} and \cite{LL}) and the Pohozaev identity of integral form (cf.
\cite{cdm} and \cite{Xu07}).

Next, we investigate the regularity of positive solutions of
(\ref{A2}) and \eqref{A1}. We have the following theorems
(cf. Theorems \ref{th5} and \ref{th11}).

\begin{theorem}\label{th5}
    Assume $u \in \mathcal{D}$ is a positive distribution solution of
    (\ref{A2}), then $u$ is differentiable in $\mathbb{R}^3$
    and belongs to $L^{q+1}(\mathbb{R}^3)$.
\end{theorem}

Theorem \ref{th*1} (ii) is the corollary of the Pohozaev identity
and Theorem \ref{th5}. Clearly, the distribution solution
$u \in \mathcal{D}$ is a finite energy solution. Now, improper
integrals in the first and the second terms of $E(u)$ (cf. \eqref{wei}) are convergent.
Theorem \ref{th5} shows that the improper integral in the third term
is also convergent. Therefore, we can avoid the tedious calculations
involved in deriving the Pohozaev identity from \eqref{A2} and directly
derive this identity from \eqref{wei} by the ideas in \cite{IMN}.

Noticing Theorem \ref{theq} (ii), we are also concerned about the
regularity of $L^{3(q-1)/2}(\mathbb{R}^3)$-solutions of \eqref{A1}.
Clearly, Theorems \ref{th*1} and \ref{th1} ensure $3(q-1)/2>1$.

First, we consider the optimal integrability interval of these
positive solutions by the regularity lifting lemma. Here the ideas in
\cite{CLOO} are used. Next, applying the method of moving planes
in integral form introduced in \cite{CLO}, we prove the radial
symmetry of these solutions. Based on these results,
we estimate the decay rate of these solutions when $|x| \to \infty$.

\begin{theorem}\label{th11}
    Assume $u \in L^{3(q-1)/2}(\mathbb{R}^3)$ is a positive
    solution of (\ref{A1}). Then

    (i) $u \in L^s(\mathbb{R}^3)$ for any $s \in (3,\infty]$.
    Moreover, the left end point $3$ is optimal.

    (ii) $u(x)$ is radially symmetric and decreasing
    about some point $x_0 \in \mathbb{R}^3$.

    (iii) There exists constant $C>1$ such that for
    sufficiently large $|x|$,
    \begin{equation}
        C^{-1} \mathcal{K}_a(x) \leq u(x) \leq C\mathcal{K}_a(x).
    \end{equation}
\end{theorem}

\begin{rem}
Theorem \ref{tha2} shows that \eqref{A1} has two super solutions
decaying fast and slowly respectively. Here, we can find $C>1$
such that for large $|x|$, there holds $(C|x|)^{-1} \leq
\mathcal{K}_a(x) \leq C|x|^{-1}$. Therefore, the positive solution in
Theorem \ref{th11} (iii) decays fast.
\end{rem}

\begin{rem}
According to Theorem \ref{theq}, we know that if $u \in
\mathcal{D}$ is a positive distribution solution of \eqref{A2}, then
(i)-(iii) in Theorem \ref{th11} still hold true. In addition, if
$u \in L^{3(q-1)/2}(\mathbb{R}^3)$ is a positive solution of \eqref{A1},
then Theorem \ref{th5} shows that $u$ is still differentiable.
\end{rem}

\vskip 5mm

Finally, we consider the Allen-Cahn-type equation \eqref{sf3} with $\sigma=1$
\begin{equation}\label{jing}
(-\Delta+a^2\Delta^2)u=u(1-u^2).
\end{equation}
More generally, we study the following
integral equation with the Coulomb potential
\begin{equation}\label{A3}
    u(x)=l+ C_\ast\int_{\mathbb{R}^3}\mathcal{K}_a(x-y)[u^{q-1}(y)(1-u^q(y))]dy,
\end{equation}
where $u:\mathbb{R}^3 \to \mathbb{R}$ with $a>0$, $l \in \mathbb{R}$,
$C_\ast>0$ and $q>1$.
If $C_\ast=(4\pi)^{-1}$ and $u^{q-1}(1-u^q)$ is rapidly decreasing,
from \eqref{A3} and \eqref{peng}
we can deduce a PDE of the Allen-Cahn-type (in distribution sense)
\begin{equation}\label{fan}
(-\Delta+a^2\Delta^2)u(x)=\delta_x \ast [u^{q-1}(1-u^q)]
=u^{q-1}(x)(1-u^q(x)).
\end{equation}
When $q=2$, \eqref{fan} is reduced to \eqref{jing}.

Here, we still use the Pohozaev identity to study the Liouville theorem.
However, although \eqref{fan} also has the variational structure
$$
E_{AC}(u)=\frac{1}{2}\int_{\mathbb{R}^3}|\nabla u|^2dx
+\frac{a^2}{2}\int_{\mathbb{R}^3}|\Delta u|^2dx
+\frac{1}{2q}\int_{\mathbb{R}^3}(1-u^q)^2 dx,
$$
not all the improper integrals in this
functional converge even if the solution is a finite energy
solution. In addition, if multiplying \eqref{fan} by $(x \cdot \nabla u)$,
we can observe that some integrals of the third order
partial derivatives will appear in the Pohozaev identity.
To avoid those tedious calculations and the stronger structural conditions
which ensure those calculations make sense, we directly apply the
Pohozaev identity in integral form which is deduced by \eqref{A3}.
Therefore, we here only study bounded solutions of \eqref{A3}
instead of distribution solutions of \eqref{fan}.

\vskip 5mm

First, we provide the value of $l$.

\begin{theorem}\label{th77}
    Assume that a uniformly continuous function $u$ solves (\ref{A3}).
    If
    \begin{equation}\label{80}
        0\leq u \leq 1 \quad on \;\mathbb{R}^3;
    \end{equation}
    and
    \begin{equation}\label{81}
        \int_{\mathbb{R}^3}u^{q-1}(1-u^q)dx <\infty.
    \end{equation}
    Then, one of the following two results holds

    (i) $u \in L^{q-1}(\mathbb{R}^3)$ and $l=0$;

    (ii) $1-u^q \in L^1(\mathbb{R}^3)$ and $l=1$.
\end{theorem}

    Next, we state the Liouville theorem.

    \begin{theorem}\label{th78}
    Assume that $u$ is a uniformly continuous and
    differentiable solution of (\ref{A3}).
If \eqref{80} and \eqref{81} hold, then

        (i) when (i) in Theorem \ref{th77} happens and $1<q \leq 6$,
        we have $u \equiv 0$;

        (ii) when (ii) in Theorem \ref{th77} happens, we have $u \equiv 1$.
    \end{theorem}

\begin{rem} 
When $a=0$ and $q=2$, \eqref{A3} becomes the equation
$u(x)=l+[|x|^{-1} \ast u(1-|u|^2)](x)$.
The Liouville theoerm can be found in \cite{LCL}.
Other related results of the Allen-Cahn-type
integral equation can be seen in \cite{cl,cl2}.
\end{rem}

This paper is arranged as follows. Theorem \ref{theq} (i) is
proved in \S2 (cf. Remark \ref{rem2.3}). Theorem \ref{theq} (ii)
is proved in \S5. Theorem \ref{th*1} (i) is proved in \S3 (cf.
Remark \ref{rem2.1}). Theorem \ref{th*1} (ii) is proved in \S2
(cf. Theorem \ref{th2.*3}). Theorem \ref{th1} (i) is proved in \S3
(cf. Theorems \ref{tha1} and \ref{tha2}). Theorem \ref{th1} (ii)
is proved in \S6. Theorem \ref{th5} is proved in \S2 and \S4.2 (cf.
Remark \ref{lem4.1}). Theorem \ref{th11} is proved
in \S4.1, \S4.3 and \S4.4. Theorems \ref{th77} and \ref{th78}
are proved in \S7.

\section{Properties of distribution solutions}

In this section, we give several properties of distribution
solutions of \eqref{A2}. The following estimates of $\mathcal{K}_a$
is needed.

First, by (3.4) and (3.5) in \cite{PS}, we have
\begin{equation}\label{chu}
\nabla \mathcal{K}_a(x) =-\frac{x}{|x|^3} +\frac{x}{|x|^3}
\left(\frac{|x|}{a}+1\right)e^{-|x|/a},
\end{equation}
and
\begin{equation}\label{feng}
\Delta \mathcal{K}_a(x) =-\frac{e^{-|x|/a}}{a^2|x|}.
\end{equation}
Moreover, by (3.6) and (3.7) in \cite{PS}, we also have
\begin{equation}\label{wang}
\begin{cases}
|\nabla \mathcal{K}_a(x)|
\leq C|x|^{-2}(1-e^{-|x|/a}
-|x|e^{-|x|/a}/a) \leq C, \quad when \quad |x| \quad small;\\[3mm]
|\nabla \mathcal{K}_a(x)|
\leq C|x|^{-2}[1+e^{-|x|/a}
(1+|x|/a)] \leq C|x|^{-2}, \quad when \quad |x| \quad large.
\end{cases}
\end{equation}

\begin{theorem}\label{th8.1}
    Assume $u \in \mathcal{D}$ is a positive
    distribution solution of (\ref{A2}),
then $u\in L^{q+1}(\mathbb{R}^3)$.
\end{theorem}

\begin{proof}
Here the standard elliptic estimation is applied.
Integrating by parts of the left hand side of \eqref{weak}, we get
\begin{equation}\label{wu}
\int_{\mathbb{R}^3}(\nabla u \nabla \phi+a^2\Delta u\Delta \phi)dx
=\int_{\mathbb{R}^3}u^q \phi dx, \quad \forall \; \phi \in C_0^\infty(\mathbb{R}^3).
\end{equation}

Take smooth function $ \zeta (x) $ satisfying
	\begin{equation*}
		\begin{cases}
			\zeta (x)=1,&\text{ $ for \left | x \right | \le 1; $ } \\
			\zeta (x) \in [0,1],&\text{ $ for \left | x \right | \in [1,2]; $ } \\
			\zeta (x)=0,&\text{ $ for \left | x \right | \ge 2, $ }
		\end{cases}
	\end{equation*}
	and write the cut-off function
	$\zeta _R(x)=\zeta({x}/{R}).$
	
	For convenience, hereafter we denote $B_{R}(0)$ by $B_{R}$ for $R>0$.
Since $C_0^\infty(\mathbb{R}^3)$ is dense in $\mathcal{D}$
(cf. Lemma 3.2 in \cite{PS}), we can take
     $\phi=u\zeta^2_R$ in (\ref{wu}) to get
	\begin{equation}\label{53}
		\int_{B_{2R}}\nabla u \nabla (u\zeta^2_R)dx
+a^2\int_{B_{2R}} \Delta u\cdot \Delta(u\zeta^2_R)dx
		=\int_{B_{2R}}u^{q+1}\zeta^2_Rdx.
	\end{equation}
	Integrating by parts and using the Cauchy inequality, we obtain
	\begin{equation}\label{yo}
		\begin{aligned}
			\int_{B_{2R}}u^{q+1}\zeta^2_Rdx
			\leq & C\int_{B_{2R}}{\left |\nabla u \right |}^2\zeta^2_Rdx
            +C\int_{B_{2R}}|\Delta u|^2\zeta_R^2dx\\
			&+C\int_{B_{2R}}|\nabla u|^2|\nabla \zeta_R|^2dx
            +C\int_{B_{2R}}u^2{\left | \nabla \zeta_R \right |}^2dx\\
			&+C\int_{B_{2R}}|\Delta u|^2|\nabla\zeta_R|^2dx
            +C\int_{B_{2R}}u^2{\left | \Delta \zeta_R \right |}^2dx.
		\end{aligned}
	\end{equation}
	In view of $|\nabla \zeta_R| \leq
CR^{-1}$ and $|\Delta \zeta_R| \leq CR^{-2}$, there hold
	\begin{equation}\label{60}
		\int_{B_{2R}}|\nabla u|^2|\nabla \zeta_R|^2dx
        \leq \frac{C}{R^2}\int_{B_{2R}}|\nabla u|^2dx,
	\end{equation}
	
	\begin{equation}\label{61}
					\int_{B_{2R}}u^2{\left | \nabla \zeta_R \right |}^2dx
			\leq C\left(\int_{B_{2R}} u^6dx\right)^{\frac{1}{3}},
	\end{equation}
		and
	\begin{equation}\label{63}
					\int_{B_{2R}}u^2{\left | \Delta \zeta_R \right |}^2dx
			\leq \frac{C}{R^2}\left(\int_{B_{2R}} u^6dx\right)^{\frac{1}{3}}.
	\end{equation}
	Inserting \eqref{60}-\eqref{63} into (\ref{yo}) and using
$u\in \mathcal{D}$, we deduce that
		$$
	\int_{B_{2R}}u^{q+1}\zeta^2_Rdx \le C,
	$$
	where $C>0$ is independent of $R$.
	Letting $R \to \infty$ we see
	$$
	\int_{\mathbb{R}^3} u^{q+1}(x)dx<\infty.
	$$
    This completes the proof of Theorem \ref{th8.1}.
\end{proof}

Now, we prove Theorem \ref{th*1} (ii).

\begin{theorem}\label{th2.*3}
If (\ref{A2}) has a positive distribution solution
$u \in \mathcal{D}$, then $q>5$.
\end{theorem}

\begin{proof}
Let $u \in \mathcal{D}$ a be positive distribution solution.

Since $C_0^\infty(\mathbb{R}^3)$ is dense in $\mathcal{D}$,
we can take $\phi=u$ in \eqref{wu} to get
\begin{equation}\label{qian}
G_3= G_1+G_2,
\end{equation}
where
$$
G_1=\|\nabla u\|_{L^2(\mathbb{R}^3)}^2, \quad G_2=a^2\|\Delta
u\|_{L^2(\mathbb{R}^3)}^2, \quad G_3=\|u\|_{L^{q+1}(\mathbb{R}^3)}^{q+1}.
$$
Clearly, $G_3>0$. We claim that $G_i>0$ for $i=1,2$. In fact, if
$G_1=0$, then $u$ is a positive constant and hence $u \not\in
L^{q+1}(\mathbb{R}^3)$. This contradicts with Theorem \ref{th8.1}. If
$G_2=0$, then $u>0$ is a harmonic function. By the Liouville
theorem, $u$ is still a positive constant. This is also contradicts
with the integrability of $u$.

According to Theorem \ref{th8.1}, three improper integrals in
\eqref{wei} are convergent. For $\mu>0$, we have
$$
E(u(\frac{x}{\mu}))=\frac{\mu}{2}\int_{\mathbb{R}^3}|\nabla u|^2dx
+\frac{a^2}{2\mu}\int_{\mathbb{R}^3}|\Delta u|^2dx
-\frac{\mu^3}{q+1}\int_{\mathbb{R}^3}u^{q+1}dx.
$$
In view of \eqref{wu}, the distribution solution $u$ is the
critical point of $E(u)$. Therefore, we have the following
Pohozaev identity
\begin{equation*}
0=\left[\frac{d}{d\mu}E(u(\frac{x}{\mu}))\right]_{\mu=1}
=\frac{G_1}{2}-\frac{G_2}{2}-\frac{3 G_3}{q+1},
\end{equation*}
Inserting \eqref{qian} into this result yields
$$
\left(\frac{1}{2}-\frac{3}{q+1}\right)
=\left(\frac{1}{2}+\frac{3}{q+1}\right)G_2 G_1^{-1}>0,
$$
which implies $q>5$.
\end{proof}

A direct corollary of Theorem \ref{th2.*3} is the following
result.

\begin{theorem} \label{th8.2}
    If $u \in \mathcal{D}$ is a positive distribution
    solution of (\ref{A2}),
    then $u \in L^{3(q-1)/2}(\mathbb{R}^3)$.
\end{theorem}

\begin{proof}
According to Theorem \ref{th2.*3}, we have $q>5$, which implies
$3(q-1)/2 \in [6,\infty]$. On the other hand, noting $u \in
\mathcal{D}$, we deduce from \eqref{zhao} that $u$ belongs to
$L^{3(q-1)/2}(\mathbb{R}^3)$.
\end{proof}

Now, we consider conclusion (i) in Theorem \ref{theq}. First we
point out that the right hand side of \eqref{A1} makes sense.

\begin{theorem} \label{th2.*2}
    If $u \in \mathcal{D}$ is a positive distribution
    solution of (\ref{A2}),
    then $\mathcal{K}_a(x-\cdot)u^q(\cdot) \in L^{1}(\mathbb{R}^3)$
    for a.e. $x \in \mathbb{R}^3$.
\end{theorem}

\begin{proof}
Clearly, \eqref{zhao} implies $u \in L^\infty(\mathbb{R}^3)$.

We observe that the defects of the improper integral
$$
\int_{\mathbb{R}^3}\mathcal{K}_a(x-y)u^q(y)dy
$$
may happen at $x$ or $\infty$. When $y$ is near $x$, for small
$\delta \in (0,1/2)$, $\mathcal{K}_a(x-y) \leq C$ as long as
$y \in B_\delta(x)$. In fact, by the
L'Hospital theorem,
    \begin{equation}\label{shen}
    \lim_{\delta \to 0}\frac{1-e^{-\frac{|x-y|}{a}}}{|x-y|}=\frac{1}{a},
    \quad when \ \ |x-y|<\delta.
    \end{equation}
Therefore,
$$
\int_{B_\delta(x)}\mathcal{K}_a(x-y)u^q(y)dy <\infty.
$$
When $y$ is near $\infty$, $\mathcal{K}_a(x-y) \leq |x-y|^{-1}$.
Thus, by Theorems \ref{th8.1} and \ref{th2.*3}, there holds
$$
\int_{\mathbb{R}^3 \setminus B_R(0)}\mathcal{K}_a(x-y)u^q(y)dy
\leq C\|u\|_{L^{q+1}(\mathbb{R}^3)}^q
\left(\int_R^\infty r^{3-(q+1)}\frac{dr}{r}\right)^{\frac{1}{q+1}}
<\infty.
$$
Combining these two estimates we complete the proof of Theorem
\ref{th2.*2}.
\end{proof}

We call $f(x)$ a double bounded function, if there exists $C>1$
such that
$$
C^{-1} \leq f(x) \leq C
$$
holds for all $x \in \mathbb{R}^3$.

\begin{theorem}\label{th2.*1}
If $u \in \mathcal{D}$ is a positive distribution solution of (\ref{A2}),
then $u$ satisfies integral equation \eqref{A1} a.e. in $\mathbb{R}^3$.
\end{theorem}

\begin{proof}
{\it Step 1.}
    For any $\psi \in C_0^{\infty}(\mathbb{R}^3)$, we write
    \begin{equation}\label{eq1}
        \Phi (x)=[\mathcal{K}_a \ast \psi](x)
        =\int_{\mathbb{R}^3}\frac{1-e^{-\frac{|x-y|}{a}}}{|x-y|}\psi (y)dy.
    \end{equation}
We claim that
\begin{equation}\label{yi}
  \Phi \in \mathcal{D}.
\end{equation}

In fact, in view of $\mathcal{K}_a(x) \leq |x|^{-1}$,
by the classical Hardy-Littlewood-Sobolev inequality
(cf. Theorem 1 in Chapter 5 of \cite{S}), we have
\begin{equation}\label{cao}
\Phi \in L^t(\mathbb{R}^3), \quad \forall \ t>3.
\end{equation}

In addition, we claim $\|\nabla\Phi\|_{L^2(\mathbb{R}^3)}
+\|\Delta\Phi\|_{L^2(\mathbb{R}^3)}<\infty$.

In fact, by Lemma 3.3 in \cite{PS}, the distributional derivatives of $\mathcal{K}_a$
\begin{equation}\label{kong}
   \nabla \Phi=(\nabla \mathcal{K}_a)\ast \psi \quad and \quad
    \Delta \Phi=(\Delta \mathcal{K}_a) \ast \psi \quad a.e. \; in \; \mathbb{R}^3.
   \end{equation}

Take $R>1$ large such that $\psi(x)=0$ as $|x| \geq R$.

    By \eqref{kong} and \eqref{wang} we can find $\delta \in (0,1/2)$
    suitably small such that when $|x|$ is sufficiently small,
    $$
        \begin{aligned}
            |\nabla \Phi|=&\int_{B_R}|\nabla \mathcal{K}_a(x-y)|\psi(y)dy \\
            \leq &C\int_{B_{\delta}}\psi(y)dy
    +C\int_{B_R\setminus B_{\delta}}|\nabla \mathcal{K}_a(x-y)|\psi(y)dy
    \leq C.
        \end{aligned}
    $$
    Similarly, when $|x|$ is sufficiently large, there holds
    $$
        |\nabla \Phi|=\int_{B_R}|\nabla \mathcal{K}_a(x-y)|\psi(y)dy
        \leq C\int_{B_R}\frac{\psi(y)dy}{|x-y|^2}\leq \frac{C}{|x|^2}.
    $$
    When $x$ is double bounded, by \eqref{kong} and \eqref{chu} we have
    $|\nabla \Phi| \leq C$.
    Combining these results, we derive that
    $$
    \int_{\mathbb{R}^3}|\nabla \Phi|^2dx
    \leq C+C\int_{\mathbb{R}^3 \setminus B_R}\frac{dx}{|x|^4}
        <\infty.
    $$
    Namely, $\nabla \Phi \in L^2(\mathbb{R}^3)$, which and \eqref{cao} lead to
    \begin{equation}\label{hua}
    \Phi \in D^{1,2}(\mathbb{R}^3).
    \end{equation}

    Finally, we use \eqref{kong} and \eqref{feng} to estimate
    $\Delta \Phi$.

    When $|x|$ is sufficiently large, it follows
    $$
    |\Delta \Phi|=\int_{B_R}\frac{e^{-\frac{|x-y|}{a}}}{a^2|x-y|}\psi(y)dy
    \leq \frac{C e^{-\frac{|x|}{a}}}{|x|}.
    $$
    When $|x|$ is sufficiently small, noting $|x-y|\geq {|x|}/{2}$
    as $y\in {B_R}\setminus B_{{|x|}/{2}}(x)$, we derive that
    \begin{equation}\label{302}
        \begin{aligned}
            |\Delta \Phi|=&\int_{B_{\frac{|x|}{2}}(x)}
    \frac{e^{-\frac{|x-y|}{a}}}{a^2|x-y|}\psi(y)dy+
            \int_{{B_R}\setminus B_{\frac{|x|}{2}}(x)}
    \frac{e^{-\frac{|x-y|}{a}}}{a^2|x-y|}\psi(y)dy\\
            \leq& C\int_0^{\frac{|x|}{2}}rdr+\frac{C e^{-\frac{|x|}{2a}}}{|x|}
                    \leq \frac{C}{|x|}.
        \end{aligned}
    \end{equation}
    When $x$ is double bounded, there holds
    $$
    |\Delta \Phi|=\int_{B_R \cap B_{\delta}(x)}
    \frac{e^{-\frac{|x-y|}{a}}}{a^2|x-y|}\psi(y)dy+
            \int_{{B_R}\setminus B_{\delta}(x)}
    \frac{e^{-\frac{|x-y|}{a}}}{a^2|x-y|}\psi(y)dy
    \leq C\int_0^{\delta}rdr+\frac{C e^{-\delta/a}}{\delta}
            \leq C.
            $$
        Combining these results, we get
        $$
                    \int_{\mathbb{R}^3}|\Delta \Phi|^2dx
    \leq C\int_{B_{\delta}(0)}\frac{dx}{|x|^2}+C
    +C\int_{\mathbb{R}^3 \setminus B_R(0)}\frac{e^{-\frac{2|x|}{a}}}{|x|^2}dx
            < \infty.
                $$
        Namely, $\Delta \Phi \in L^2(\mathbb{R}^3)$.
    This and \eqref{hua} imply \eqref{yi}.

{\it Step 2.}
According to Lemma 3.3 in \cite{PS}, there holds
$$
-\Delta \Phi+a^2{\Delta}^2\Phi=\psi
$$
in distribution sense. Namely,
\begin{equation}\label{jia}
  \int_{\mathbb{R}^3}\Phi(-\Delta+a^2\Delta^2)\phi dx
  =\int_{\mathbb{R}^3}\psi\phi dx,
  \quad \forall \phi \in C_0^\infty(\mathbb{R}^3).
\end{equation}
In view of \eqref{yi}, we can integrate by parts the left hand side of
\eqref{jia}. Therefore, from \eqref{jia} we get
\begin{equation*}
  \int_{\mathbb{R}^3} (\nabla \Phi \nabla \phi
  +a^2 \Delta \Phi \Delta \phi) dx
  =\int_{\mathbb{R}^3}\psi\phi dx,
  \quad \forall \phi \in C_0^\infty(\mathbb{R}^3).
\end{equation*}
Since $C_0^\infty(\mathbb{R}^3)$ is dense in $\mathcal{D}$, we can take
$\phi=u$ in the above result to obtain
\begin{equation}\label{ding}
  \int_{\mathbb{R}^3} u\psi dx
  =\int_{\mathbb{R}^3} (\nabla \Phi \nabla u +a^2 \Delta \Phi \Delta u) dx.
\end{equation}
Similarly, noting \eqref{yi} we can also take $\phi=\Phi$ in \eqref{wu} to get
$$
\int_{\mathbb{R}^3}u^q \Phi dx
  =\int_{\mathbb{R}^3}(\nabla u \nabla \Phi+a^2\Delta u\Delta \Phi)dx.
$$
Combining this result with \eqref{ding} and using \eqref{eq1}, we
obtain
$$
\begin{aligned}
\int_{\mathbb{R}^3}u(x)\psi(x)dx=&\int_{\mathbb{R}^3} u^q(x)\Phi(x)dx\\
=&\int_{\mathbb{R}^3}u^q(x)\int_{\mathbb{R}^3}
\frac{1-e^{-\frac{|x-y|}{a}}}{|x-y|}\psi (y)dydx\\
=&\int_{\mathbb{R}^3}\psi(x)\int_{\mathbb{R}^3}
\frac{1-e^{-\frac{|x-y|}{a}}}{|x-y|}u^q(y)dydx.
\end{aligned}
$$
By Theorem \ref{th2.*2}, the right hand side of the result above
makes sense. This shows that $u$ satisfies (\ref{A1}) almost
everywhere in $\mathbb{R}^3$.
\end{proof}

\begin{rem}\label{rem2.2}
By an analogous argument above, we also see that $u$ is a positive
super solution of \eqref{A1} if $u \in \mathcal{D}$ is a positive
super solution of (\ref{A2}) in distribution sense. Here we call
$u$ a positive super solution of \eqref{A1}, if
$$
u(x) \geq \int_{\mathbb{R}^3}\mathcal{K}_a(x-y)u^q(y)dy,
\quad a.e. \ \ in \ \mathbb{R}^3.
$$
\end{rem}

\begin{rem}\label{rem2.3}
Theorems \ref{th8.2} and \ref{th2.*1} imply the first conclusion
in Theorem \ref{theq}.
\end{rem}

\section{Existence of super solutions}

In this section, we prove Theorem \ref{th1} (i).

\begin{theorem}\label{tha1}
When $0<q \leq 3$, (\ref{A1}) has no positive super solution in
$L_{\rm loc}^\infty(\mathbb{R}^3)$.
\end{theorem}

\begin{proof}
Let $0<q\leq 3$. We will deduce a contradiction if $u$ is a
positive super solution of (\ref{A1}).

In view of $0<u \in L_{\rm loc}^\infty(\mathbb{R}^3)$, we can find $C>0$
such that as $|y|<1$, there holds $u(y) \geq C$. Therefore, we deduce that
for large $|x|$,
\begin{equation}\label{chen}
u(x) \geq \int_{|y|\leq1}\frac{1-e^{-\frac{|x-y|}{a}}}{|x-y|}u^q(y)dy
\geq C\int_{|y|\leq 1} \frac{1-e^{-\frac{|x-y|}{a}}}{|x-y|}dy \geq \frac{C}{|x|}.
\end{equation}
By this estimate, we have that, for large $|x|$,
    \begin{equation}\label{dd}
u(x) \geq
\int_{B_{\frac{|x|}{2}}(x)}\frac{1-e^{-\frac{|x-y|}{a}}}{|x-y|}u^q(y)dy
\geq  \frac{C}{|x|^q}\int_{0}^{\frac{|x|}{2}}r^3 \cdot
\frac{1-e^{-\frac{r}{a}}}{r} \frac{dr}{r} \geq
\frac{C}{|x|^{q-2}}.
    \end{equation}

When $0<q \leq 2$, we have $2-q(q-2)>0$. Therefore, by \eqref{dd}
we have
$$
u(x) \geq \int_{\mathbb{R}^3 \setminus
B_{2|x|}(0)}\frac{1-e^{-\frac{|x-y|}{a}}}{|x-y|}u^q(y)dy \geq
C\int_{\mathbb{R}^3 \setminus B_{2|x|}(0)}\frac{dy}{|y|^{1+q(q-2)}}
=\infty.
$$
This is impossible because of $u \in L_{\rm loc}^\infty(\mathbb{R}^3)$.

When $2<q<3$, we see
\begin{equation}\label{geng}
q-2-\frac{2}{q-1}<0.
\end{equation}
    By (\ref{dd}), for large $|x|$, we have
    $$
u(x) \geq \int_{B_{\frac{|x|}{2}}(x)}
\frac{1-e^{-\frac{|x-y|}{a}}}{|x-y|}u^q(y)dy \geq
\frac{C}{|x|^{q(q-2)}} \int_{0}^{\frac{|x|}{2}}r^3 \cdot
\frac{1-e^{-\frac{r}{a}}}{r} \frac{dr}{r} \geq
\frac{C}{|x|^{q(q-2)-2}}.
    $$
Set $a_0=1$, $a_1=q-2$, $a_2=q(q-2)-2$. By the estimate above, for
large $|x|$, we have the following iteration result
    $$
    u(x) \geq \frac{C}{|x|^{a_n}},
    $$
    where $a_n=qa_{n-1}-2$. That is,
    $$
    a_n=q^{n-1}(q-2-\frac{2}{q-1})+\frac{2}{q-1}.
    $$
In view of \eqref{geng}, we can find large $n$ such that
$a_n<2/q$. Now, we see that
    $$
u(x) \geq  \int_{\mathbb{R}^3 \setminus
B_{2|x|}(0)}\frac{1-e^{-\frac{|x-y|}{a}}}{|x-y|}u^q(y)dy \geq
C\int_{\mathbb{R}^3 \setminus B_{2|x|}(0)} \frac{dy}{|y|^{qa_n+1}} =\infty.
    $$
    This is impossible.

    When $q=3$, we also can deduce a contradiction as follows.

   In view of \eqref{shen}, for large $R>0$ and small $\delta \in (0,1/2)$,
    $$
u(x) \geq \frac{1}{2a}\int_{B_R \cap
B_{\delta}(x)}u^3(y)dy+\frac{1-e^{-\frac{\delta}{a}}}{2R}\int_{B_R
\setminus B_{\delta}(x)}u^3(y)dy,  \quad  \forall x \in B_R.
$$
Let $R$ be suitably large, then ${1}/{a}>(1-e^{-{\delta}/{a}})/R$,
and hence
    $$
    u(x) \geq \frac{C}{R}\int_{B_R }u^3(y)dy.
    $$
    Thus, we have
    $$
    u^3(y) \geq \frac{C}{R^3}\left(\int_{B_R}u^3(y)dy\right)^3.
    $$
    Integrating on $ B_{R}$ yields
    \begin{equation}\label{hutu}
    \int_{B_R }u^3(x)dx \geq C\left(\int_{B_R}u^3(y)dy\right)^3.
    \end{equation}
Here $C>0$ is independent of $R$. Since
$u$ is a positive solution, the result above implies
    \begin{equation}\label{ww}
        u \in L^3(\mathbb{R}^3) .
    \end{equation}

    For $x \in B_R\setminus B_{\frac{R}{2}}$, we get
    $$
u(x) \geq
\int_{B_{\frac{R}{3}}}\frac{1-e^{-\frac{|x-y|}{a}}}{|x-y|}u^3(y)dy
\geq \frac{C}{R}\int_{B_{\frac{R}{3}}}u^3(y)dy.
    $$
Similar to the derivation of \eqref{hutu}, cubing the above result
and integrating on $B_R \setminus B_{\frac{R}{2}}$, we have
    $$
    \int_{B_R \setminus B_{\frac{R}{2}}}u^3(x)dx
    \geq C\left(\int_{B_{\frac{R}{3}}}u^3(y)dy\right)^3,
    $$
where $C>0$ is independent of $R$. Letting $R \to \infty$ and
using (\ref{ww}), we obtain
    $$
    \left(\int_{\mathbb{R}^3}u^3(y)dy\right)^3 = 0,
    $$
which implies $u \equiv 0$ in $\mathbb{R}^3$. This contradicts with $u>0$.
Theorem (\ref{tha1}) is proved.
\end{proof}

\begin{rem} \label{rem2.1} If $u \in \mathcal{D}$, (\ref{zhao})
implies $u \in L^\infty(\mathbb{R}^3) \subset L_{\rm loc}^\infty(\mathbb{R}^3)$.
Noting Remark \ref{rem2.2}, we know that Theorem \ref{th*1} (i) is
a corollary of Theorem \ref{tha1}.
\end{rem}

\begin{theorem}\label{tha2}
    Let $q>3$. Then both
$$
u(x)=\frac{1}{(1+|x|^2)^{1/2}} \quad       and
\quad u(x)=\frac{1}{(1+|x|^2)^{1/(q-1)}}
$$
    are positive radial super solutions of (\ref{A1}).
\end{theorem}

\begin{rem}
When $q>3$, there holds $1>2/(q-1)$. We call $1$ the fast
decay rate and $2/(q-1)$ the slow decay rate.
\end{rem}

\textbf{Proof of Theorem \ref{tha2}.}

\begin{proof}
Set
\begin{equation}\label{f1}
        u_1(x)=\frac{1}{(1+|x|^2)^{\theta}},
    \end{equation}
where $\theta>0$ will be determined later. Inserting (\ref{f1})
into the right hand side of (\ref{A1}).

Take $R>1$. We can find $C>1$ such that when $|x|>2R$, there holds
    $$\begin{aligned}
&\int_{\mathbb{R}^3}\frac{1-e^{-\frac{|x-y|}{a}}}{|x-y|}u_1^q(y)dy\\
\leq &\frac{C}{(1+|x|^2)^{\frac{1}{2}}}
\int_{B_R}\frac{dy}{(1+|y|^2)^{q\theta}}
+\frac{C}{(1+|x|^2)^{q\theta}}\int_{B_{\frac{|x|}{2}}(x)}
\frac{1-e^{-\frac{|x-y|}{a}}}{|x-y|}dy\\
&+C\int_{(B_{2|x|} \setminus {B_R})\setminus
B_{\frac{|x|}{2}}(x)}\frac{1-e^{-\frac{|x-y|}{a}}}
{|x-y||y|^{2q\theta}}dy
+C\int_{\mathbb{R}^3 \setminus {B_{2|x|}}}\frac{1-e^{-\frac{|x-y|}{a}}}
{|x-y||y|^{2q\theta}}dy.
    \end{aligned}
    $$
In view of $q>3$, we take $\theta =1/(q-1)$, which implies
$2<2q\theta<3$. Thus, for some double function $c_1(x)$, we obtain
   \begin{equation}\label{lang}
   \begin{aligned}
\int_{\mathbb{R}^3}\frac{1-e^{-\frac{|x-y|}{a}}}{|x-y|}u_1^q(y)dy
\leq &\frac{C}{(1+|x|^2)^{\frac{1}{2}}}
+\frac{C}{(1+|x|^2)^{q\theta-1}}\\
&+\frac{C}{|x|}\int_{R}^{2|x|}r^{3-2q\theta}\frac{dr}{r}
+C\int_{2|x|}^{\infty}r^{2-2q\theta}\frac{dr}{r}\\
\leq &\frac{C}{(1+|x|^2)^{q\theta-1}}.
    \end{aligned}
   \end{equation}

On the other hand, for $|x|>2R$,
$$
\int_{\mathbb{R}^3}\frac{1-e^{-\frac{|x-y|}{a}}}{|x-y|}u_1^q(y)dy
\geq \frac{1}{C(1+|x|^2)^{q\theta}}\int_{B_{\frac{|x|}{2}}(x)}
\frac{1-e^{-\frac{|x-y|}{a}}}{|x-y|}dy
\geq \frac{1}{C(1+|x|^2)^{q\theta-1}}.
$$
This result and \eqref{lang} imply that (\ref{f1}) is the slowly decaying radial
solution of
\begin{equation}\label{huyou}
c_1(x)u(x)=\int_{\mathbb{R}^3}\frac{1-e^{-\frac{|x-y|}{a}}}{|x-y|}u^q(y)dy
\end{equation}
when $|x|>2R$, where $c_1(x)$ is some double bounded function.

In addition, we take $\theta = 1/2$, which implies $2q\theta>3$ as
long as $q>3$. Therefore, we can find $C>1$ such that
    $$
    \begin{aligned}
\int_{\mathbb{R}^3}\frac{1-e^{-\frac{|x-y|}{a}}}{|x-y|}u_1^q(y)dy
\leq &\frac{C}{(1+|x|^2)^{\frac{1}{2}}}
+\frac{C}{(1+|x|^2)^{q\theta-1}}\\
&+\frac{C}{|x|}\int_{R}^{2|x|}r^{3-2q\theta}\frac{dr}{r}
+C\int_{2|x|}^{\infty}r^{2-2q\theta}\frac{dr}{r}\\
\leq &\frac{C}{(1+|x|^2)^{\frac{1}{2}}}.
    \end{aligned}
    $$
On the other hand,
$$
\int_{\mathbb{R}^3}\frac{1-e^{-\frac{|x-y|}{a}}}{|x-y|}u_1^q(y)dy
\geq \int_{B_R}\frac{dy}{(1+|y|^2)^{\frac{1}{2}}}
\geq \frac{1}{C(1+|x|^2)^{\frac{1}{2}}}.
$$
Combining these two results, we see that (\ref{f1}) is the fast
decaying radial solution of (\ref{huyou}) when $|x|>2R$.

For double bounded function $c_1(x)$, we can find a constant $C_1>0$
such that $c_1(x) \leq C_1$. Write $u_2=u_1/\tilde{C}$, where
${\tilde{C}}=C_1^{1/(q-1)}$. Since $u_1$ is the radial
solution of (\ref{huyou}), we see that
    $$
\tilde{C}u_2(x)=u_1(x)=
c_1^{-1}(x)\int_{\mathbb{R}^3}\mathcal{K}_a(x-y)u_1^q(y)dy \geq
{\tilde{C}}^qC_1^{-1}\int_{\mathbb{R}^3}\mathcal{K}_a(x-y)u_2^q(y)dy.
        $$
Therefore,
    $$
u_2(x) \geq
{\tilde{C}}^{q-1}C_1^{-1}\int_{\mathbb{R}^3}\mathcal{K}_a(x-y)u^q(y)dy =
\int_{\mathbb{R}^3}\mathcal{K}_a(x-y)u_2^q(y)dy.
    $$
Namely, $u_2$ is a super solution of (\ref{A1}) when $|x|>2R$.

When $|x| \leq 2R$, $u_1 \geq (1+4R^2)^{-\theta}$ with $\theta \in
\{1/2,1/(q-1)\}$.
By the same argument above, we can find a constant
$C_2>0$ such that
    $$
    \int_{\mathbb{R}^3}\frac{1-e^{-\frac{|x-y|}{a}}}{|x-y|}u_1^q(y)dy
    \leq C_2.
    $$
Therefore,
$$
u_1(x) \geq C_3 \int_{\mathbb{R}^3}
\frac{1-e^{-\frac{|x-y|}{a}}}{|x-y|}u_1^q(y)dy,
$$
where $C_3=[C_2(1+4R^2)^{\theta}]^{-1}$. Write $u_3=C_3^{1/(q-1)}u_1$.
Thus, $u_3$ is a super solution of \eqref{A1} when $|x| \leq 2R$.

Theorem \ref{tha2}
is proved.
\end{proof}

Combining Theorems \ref{tha1} and \ref{tha2}, we prove the first
conclusion in Theorem \ref{th1}.

\section{Regularity of integral equation}

In this section, we prove Theorem \ref{th11}. Let $u \in
L^{3(q-1)/2}(\mathbb{R}^3)$ be a positive solution of \eqref{A1}.

\subsection{Integrability}

\textbf{Proof of Theorem \ref{th11} (i).}

{\it Step 1.} We claim that $u \in L^s(\mathbb{R}^3)$ for all $s>3$.

For $A>0$, set
$$
u_A(x)=\begin{cases}
    u(x),&\text{ $ if\; u(x)>A \; or \; |x|>A; $ } \\
    0,&\text{ $ otherwise, $ }
\end{cases}
$$
Set $u_B(x)=u(x)-u_A(x)$. Let $\omega \in L^s(\mathbb{R}^3)$ for $s>3$
and define the liner operator
$$
(T\omega)(x)=\int_{\mathbb{R}^3}\mathcal{K}_a(x-y)u_A^{q-1}(y)\omega(y)dy.
$$
Write
$$
F(x):=\int_{\mathbb{R}^3}\mathcal{K}_a(x-y)u_B^q(y)dy.
$$
Obviously, $u$ solves the operator equation
$$
\omega =T\omega+F.
$$
Noting that $\mathcal{K}_a(x-y) \leq |x-y|^{-1}$, applying the
classical Hardy-Littlewood-Sobolev inequality
and the H\"{o}lder inequality, we obtain
\begin{equation}\label{gj}
    \begin{aligned}
\left \| T\omega  \right \|_{L^s(\mathbb{R}^3)} \leq &C\left \|
I_2(u_{A}^{q-1}\omega)  \right \|_{L^s(\mathbb{R}^3)}
\leq C\left \| u_{A}^{q-1}\omega \right \|_{L^{\frac{3s}{3+2s}}(\mathbb{R}^3)}\\
\leq &C\left \|u_{A}\right \|_{L^{\frac{3(q-1)}{2}}(\mathbb{R}^3)}^{q-1}
\left \| \omega  \right \|_{L^s(\mathbb{R}^3)} .
    \end{aligned}
\end{equation}
Here $I_2(f)$ is the Newton potential of $f$. In view of $u \in
L^{\frac{3(q-1)}{2}}(\mathbb{R}^3)$, we get
$$
C \left \| u_A  \right \|_{L^{\frac{3(q-1)}{2}}(\mathbb{R}^3)}^{q-1}
\leq \frac{1}{2},
$$
when $A>0$ is sufficiently large. Hence, $T$ is a shrinking
operator. Since $T$ is linear, it is also a contraction map from
$L^s(\mathbb{R}^3)$ to itself as long as $s>3$.

Similar to (\ref{gj}), for any $s>3$, there holds
$$
\left \| F  \right \|_{L^s(\mathbb{R}^3)}
\leq  C\left \| u_{B}^q \right \|_{L^{\frac{3s}{3+2s}}(\mathbb{R}^3)}.
$$
Combining with the definition of $u_B$, we have $F \in L^s(\mathbb{R}^3)$.
By virtue of the regularity lifting lemma (cf. Theorem 3.3.1 in
\cite{CL-Book} or Lemma 2.2 in \cite{MCL}), we deduce that $u \in
L^s(\mathbb{R}^3)$ for any $s>3$.

{\it Step 2.} We claim that $u \in L^\infty(\mathbb{R}^3)$.

In fact, Step 1 shows
    \begin{equation}\label{q3}
        u \in L^t(\mathbb{R}^3), \quad \forall \; t>3.
    \end{equation}
In view of (\ref{A1}) and $\mathcal{K}_a(x-y) \leq |x-y|^{-1}$,
there holds
    \begin{equation}\label{2.01}
        u(x)=[\mathcal{K}_a \ast u^q](x) \leq CI_2(u^q)(x).
    \end{equation}
By exchanging integral variables, we convert the Newton potential
into the Wolff type potential. That is,
    \begin{equation}\label{2.02}
I_2(u^q)(x)  =\int_{\mathbb{R}^3}\frac{u^q(y)}{|x-y|}dy =\int_{\mathbb{R}^3}
u^q(y)\int_{|x-y|}^{\infty}t^{-1}\frac{dt}{t}dy
=\int_{0}^{\infty}\left(\frac{\int_{B_t(x)}u^q(y)dy}{t}\right)\frac{dt}{t}.
            \end{equation}
    Set
    $$
I_2(u^q)(x)=\int_{0}^{1}\left(\frac{\int_{B_t(x)}u^q(y)dy}{t}\right)\frac{dt}{t}
+\int_{1}^{\infty}\left(\frac{\int_{B_t(x)}u^q(y)dy}{t}\right)\frac{dt}{t}
:=J_1+J_2.
    $$

Take $l>3/2$. According to Theorem \ref{tha1}, $q>3$.
Therefore, \eqref{q3} implies $u
\in L^{ql}(\mathbb{R}^3)$. Then by the H\"{o}lder
inequality, we have
    $$
J_1 \leq C\int_{0}^{1}\frac{1}{t} \left\|u^q
\right\|_{L^l(\mathbb{R}^3)}t^{3(1-\frac{1}{l})} \frac{dt}{t} \leq
C\int_{0}^{1}t^{2-\frac{3}{l}}\frac{dt}{t} \leq C.
        $$
For $\delta \in (0, 1)$ and $z \in B_{\delta}(x)$, we have $B_t(x)
\subset B_{t+\delta}(z)$. Therefore, by \eqref{2.02}, it follows that
    $$
    J_2 \leq
\int_{1}^{\infty}\left(\frac{\int_{B_{t+\delta}(z)}u^q(y)dy}{t+\delta}\right)
 \left(\frac{t+\delta}{t}\right)^2\frac{d(t+\delta)}{t+\delta}
\leq  2^2\int_{1+\delta}^{\infty}\frac{\int_{B_{t}(z)}u^q(y)dy}{t}\frac{dt}{t}
\leq  CI_2(u^q)(z).
        $$
Combining estimates of $J_1$ and $J_2$ with \eqref{2.01}, we get
    $$
    u(x) \leq CI_2(u^q)(x) \leq C+CI_2(u^q)(z), \quad when \ \ z \in B_{\delta}(x).
    $$
    Take $t>3$. Then
$$
    u^t(x) \leq C+C[I_2(u^q)(z)]^t, \quad when \ \ z \in B_{\delta}(x).
    $$
Integrating on $B_{\delta}(x)$ and multiplying by
$|B_{\delta}|^{-1}$, and using the classical Hardy-Littlewood-Sobolev
inequality, we obtain
    $$\begin{aligned}
u(x) =&\left(|B_{\delta}|^{-1}\int_{B_{\delta}(x)}u^t(x)dz\right)^{1/t}\\
\leq & C+ C
\left(\int_{B_{\delta}(x)}[I_2(u^q)(z)]^tdz\right)^{1/t} \\
\leq & C+C\|u\|_{L^{\frac{3qt}{3+2t}}(\mathbb{R}^3)}^q \\
\leq & C \quad \quad (by \ \ 3qt/(3+2t)>3 \ \ and  \ \ \eqref{q3}).
    \end{aligned}
    $$
    This shows that $u$ is bounded in $\mathbb{R}^3$.

{\it Step 3.}
We claim that $3$ is optimal. In fact, \eqref{chen} implies
$u(x) \geq C|x|^{-1}$ for $|x|>M$ with suitably large $M$. Therefore,
$$
\int_{\mathbb{R}^3}u^s(x)dx \geq C\int_{\mathbb{R}^3 \setminus B_M(0)}u^s(x)dx
=\infty, \quad \forall s\leq3.
$$
This completes the proof of (i) in Theorem \ref{th11}.

Moreover, by an analogous argument above,
we can deduce the following conclusion.

\begin{theorem}\label{th2.1}
Assume $u \in L^{\frac{3(q-1)}{2}}(\mathbb{R}^3)$ is a positive solution of
equation (\ref{A1}), then $u$ converges to $0$ when $|x| \to
\infty$.
\end{theorem}

\begin{proof}
According to (\ref{2.01}) and (\ref{2.02}), taking $x_0 \in \mathbb{R}^3$,
we see
    \begin{equation}\label{2.10}
u(x_0) \leq
\int_{0}^{\infty}\left(\frac{\int_{B_t(x_0)}u^q(z)dz}{t}\right)\frac{dt}{t}.
    \end{equation}

First for any $\varepsilon>0$,
there exists $\delta \in (0,1)$ such that
    \begin{equation}\label{2.12}
\int_{0}^{\delta}\left[\frac{\int_{B_t(x_0)}u^q(z)dz}{t}\right]\frac{dt}{t}
\leq \left \| u \right \| _{L^\infty(\mathbb{R}^3)}^q\int_{0}^{\delta}
t^2\frac{dt}{t} \leq C \delta^2 < \varepsilon.
    \end{equation}
    In addition, for $|x-x_0|<\delta$,
    \begin{equation}\label{2.13}
        \begin{aligned}
\int_{\delta}^{\infty}\left[\frac{\int_{B_t(x_0)}u^q(z)dz}{t}\right]\frac{dt}{t}
\leq &
\int_{\delta}^{\infty}\frac{\int_{B_{t+\delta}(x)}u^q(z)dz}{t}
\left(\frac{t+\delta}{t}\right)^2\frac{d(t+\delta)}{t+\delta}\\
\leq &
C\int_{0}^{\infty}\left[\frac{\int_{B_t(x)}u^q(z)dz}{t}\right]\frac{dt}{t}
\leq  CI_2(u^q)(x).
        \end{aligned}
    \end{equation}
Inserting (\ref{2.12}) and (\ref{2.13}) into (\ref{2.10}), we
obtain that
    \begin{equation}\label{sun}
u(x_0)<C\varepsilon+CI_2(u^q)(x), \quad for \;\; |x-x_0|<\delta.
    \end{equation}
Letting $s>3$ sufficiently large, we get
    $$
u^s(x_0)<C\varepsilon^s+C[I_2(u^q)(x)]^s,\quad for \;\;
|x-x_0|<\delta.
    $$
    Integrating on $B_{\delta}(x_0)$ and multiplying by
    $|B_\delta|^{-1}$, we have
    \begin{equation}\label{2.14}
u^s(x_0)
\leq C\varepsilon^s+C\left \| I_2(u^q)
\right \|_{L^s(B_{\delta}(x_0))}^s.
\end{equation}
In view of $s>3$, and using the classical Hardy-Littlewood-Sobolev
inequality and \eqref{q3}, we get
    $$
    \left \| I_2(u^q) \right \|_{L^s(\mathbb{R}^3)}
\leq C \| u^q\|_{L^{\frac{3s}{3+2s}}(\mathbb{R}^3)}
=C \| u \|_{L^{\frac{3sq}{3+2s}}(\mathbb{R}^3)}^q<\infty.
    $$
    This implies $ I_2(u^q)\in L^s(\mathbb{R}^3)$.
    Hence,
    $$
    \lim_{|x_0|\to \infty}\int_{B_{\delta}(x_0)}[I_2(u^q)(x)]^s dx=0.
    $$
    Combining this result and (\ref{2.14}), we deduce
    $$
    \lim_{|x_0|\to \infty}u^s(x_0)=0.
    $$
    This completes the proof of Theorem \ref{th2.1}.
\end{proof}

\subsection{Differentiability}

In this subsection, we prove Theorem \ref{th5}.

\begin{theorem}\label{th6.1}
Assume that the positive function $u \in L^{3(q-1)/2}(\mathbb{R}^3)$ solves \eqref{A1},
then $u$ is differentiable in $\mathbb{R}^3$.
\end{theorem}

\begin{proof}
According to Theorem \ref{tha1}, we know $q>3$.

{\it Step 1.}
Take $R>2$ suitably large.
We claim that
$$
        U(x):=\int_{B_{R}}\mathcal{K}_a(x-y) u^q(y)dy
        $$
        is differentiable.

In fact, an analogous argument in \S4.2 of \cite{GT} is used here. Set
    	$$
    	w_1(x):=\int_{B_R}\nabla_x \mathcal{K}_a(x-y)u^q(y)dy.
    	$$
    By\eqref{chu}, \eqref{wang} and Theorem \ref{th11} (i), we can find $\delta
\in (0,1/2)$ suitably small such that the right hand side is finite
    when $x$ belongs to $\mathbb{R}^3 \setminus B_{2R}$,
    $B_{2R} \setminus B_\delta(y)$ and $B_\delta(y)$ respectively.
    Thus, $w_1$ is well defined in $\mathbb{R}^3$.

    	Take a cut-off function $\zeta \in C^1(\mathbb{R})$ satisfying
    $0 \leq \zeta \leq 1$, $0\leq \zeta' \leq 2$, $\zeta (t)=0$ for $t\leq 1$,
    and $\zeta(t)=1$ for $t\geq 2$. For $\varepsilon>0$, write
    	$$
    	U_{\varepsilon}(x)
    :=\int_{B_R}\mathcal{K}_a(x-y)\zeta_{\varepsilon}(|x-y|)u^q(y)dy.
    	$$
    	Here $\zeta_{\varepsilon}(|x-y|)=\zeta({|x-y|}/{\varepsilon})$ and
    $U_{\varepsilon}(x)$ is well defined because it is bounded by $u(x)$.
        	In addition,
    \begin{equation}\label{952}
    		\begin{aligned}
    			&|\nabla_x[\mathcal{K}_a(x-y)\zeta_{\varepsilon}(|x-y|)]u^q(y)|\\
    		=&|[\nabla_x\mathcal{K}_a(x-y)] \zeta_{\varepsilon}(|x-y|)u^q(y)|
    +|[\nabla_x\zeta_{\varepsilon}(|x-y|)] \mathcal{K}_a(x-y)u^q(y)|\\
    		\leq &(|\nabla_x\mathcal{K}_a(x-y)|+2| \mathcal{K}_a(x-y)|)|u^q(y)|.
    		\end{aligned}
    	\end{equation}
    Let $y \in B_R$. By \eqref{shen}, \eqref{chu} and \eqref{wang},
    we can find an absolute constant $C>0$ such that
    $$
    |\nabla_x\mathcal{K}_a(x-y)|+2| \mathcal{K}_a(x-y)| \leq C
    $$
    when $x$ belongs to $\mathbb{R}^3 \setminus B_{2R}$,
    $B_{2R} \setminus B_\delta(y)$ and $B_\delta(y)$ respectively.
    Inserting this result into \eqref{952} and using \eqref{q3}, we get
    	$$
    		|\nabla_x[\mathcal{K}_a(x-y)\zeta_{\varepsilon}(|x-y|)]u^q(y)|
    		\leq  Cu^q(y) \in L^1(B_R).
    		$$
Thus,
$\nabla_x [\mathcal{K}_a(x-y)\zeta_{\varepsilon}(|x-y|)]u^q(y)$
can be dominated by an integrable function.
Therefore, we can still use the differentiation theorem
for improper integrals with parametric variables (cf. Theorem 3.16 in \cite{Ben})
to see that
$U_{\varepsilon}(x)$ is differentiable in $\mathbb{R}^3$ and
    	$$
    	\nabla U_{\varepsilon}(x)
    =\int_{B_R}\nabla_x[\mathcal{K}_a(x-y)\zeta_{\varepsilon}(|x-y|)]u^q(y)dy.
    	$$
    	Clearly, there hold
    $$
    U(x)-U_{\varepsilon}(x)
    =\int_{|x-y|\leq 2\varepsilon}(1-\zeta_{\varepsilon}) \mathcal{K}_a(x-y)u^q(y)dy,
    	$$
    and
    	$$
    	w_1(x)-\nabla U_{\varepsilon}(x)
    =\int_{|x-y|\leq 2\varepsilon}\nabla_x[(1-\zeta_{\varepsilon})
    \mathcal{K}_a(x-y)]u^q(y)dy.
    	$$
    	Noting
    $$
|U(x)- U_{\varepsilon}(x)| \leq C\|u\|_{L^\infty(B_R)}^p |B_{2\varepsilon}|,
    $$
which is implied by \eqref{shen},    and
    	\begin{equation*}
    		\begin{aligned}
    			|w_1(x)-\nabla U_{\varepsilon}(x)|
    			\leq &\int_{|x-y|\leq 2\varepsilon}(\frac{C}{\varepsilon}|\mathcal{K}_a(x-y)|
    +|\nabla_x\mathcal{K}_a(x-y)|)u^q(y)dy\\
    			\leq &\frac{C}{\varepsilon}\int_{|x-y|\leq 2\varepsilon}\frac{1}{|x-y|}u^q(y)dy
    +C\int_{|x-y|\leq 2\varepsilon}u^q(y)dy \quad (by \ \ \eqref{wang})\\
    			\leq &C\|u\|_{L^\infty(B_R)}^q(\varepsilon+|B_{2\varepsilon}|),
    		\end{aligned}
    	\end{equation*}
        we know that $U_{\varepsilon}$ and $\nabla U_{\varepsilon}$
        converge uniformly to $U$ and $w_1$ respectively as $\varepsilon \to 0$.
        Hence, the claim holds and
       \begin{equation}\label{954}
       	\nabla U(x)=\int_{B_R}\nabla_x \mathcal{K}_a(x-y)u^q(y)dy.
       \end{equation}

{\it Step 2.} Write
$$
        V(x):=\int_{\mathbb{R}^3\setminus B_R}\mathcal{K}_a(x-y) u^q(y)dy,
        \quad
V_\varepsilon(x):=\int_{\mathbb{R}^3\setminus B_R}
\mathcal{K}_a(x-y)\zeta_{\varepsilon}(|x-y|)u^q(y)dy,
    	$$
    and
$$
    	w_2(x):=\int_{\mathbb{R}^3\setminus B_R}\nabla_x \mathcal{K}_a(x-y)u^q(y)dy.
    	$$
    Clearly, $V(x)$ and $V_\varepsilon(x)$ are bounded by $u(x)$. In addition,
\eqref{chu}, \eqref{wang} and Theorem \ref{th11} (i) show that $w_2$ is finite
    when $x$ belongs to $(\mathbb{R}^3 \setminus B_{R/2}) \setminus B_\delta(y)$,
    $B_{R/2}$ and $B_\delta(y)$ respectively.
    Thus, $V(x)$, $V_\varepsilon(x)$ and $w_1$ are well defined in $\mathbb{R}^3$.

Let $y \in \mathbb{R}^3\setminus B_R$. By \eqref{chu}, \eqref{wang}
and Theorem \ref{th11} (i),
we can find an absolute constant $C>0$ such that
$$
            |\nabla_x [\mathcal{K}_a(x-y)\zeta_{\varepsilon}(|x-y|)]u^q(y)|
            \leq Cu^q(y)
            $$
when $x$ belongs to $(\mathbb{R}^3 \setminus B_{R/2}) \setminus B_\delta(y)$,
    $B_{R/2}$ and $B_\delta(y)$ respectively.
    Therefore, we can use the differentiation theorem
for improper integrals with parametric variables
(cf. Theorem 3.16 in \cite{Ben}) to
see that $V_\varepsilon(x)$ is differentiable in $\mathbb{R}^3$, and
        $$
        \nabla V_\varepsilon(x)=\int_{\mathbb{R}^3\setminus B_R}\nabla_x
        [\mathcal{K}_a(x-y)\zeta_\varepsilon(|x-y|)]u^q(y)dy.
        $$
Similarly to Step 1, we also see that $V_\varepsilon$ and $\nabla V_\varepsilon$
converge uniformly to $V$ and $w_2$ respectively as $\varepsilon \to 0$.
This shows that $V$ is also differentiable.

        Since $u$ solves \eqref{A1}, we have $u=U+V$.
        Therefore, Steps 1 and 2 show that $u$ is also differentiable
        in $\mathbb{R}^3$.
\end{proof}

\begin{rem}\label{lem4.1}
Since $u \in \mathcal{D}$ is a positive distribution solution of
(\ref{A2}), Remark \ref{rem2.3} shows that $u$ belongs to
$L^{3(q-1)/2}(\mathbb{R}^3)$ and solves \eqref{A1}.
In view of Theorems \ref{th8.1} and \ref{th6.1}, we see
that Theorem \ref{th5} is true.
\end{rem}

\subsection{Radial symmetry}

In this subsection, we prove conclusion (ii) in Theorem \ref{th11}.
According to Theorem \ref{th6.1}, $u$ is continuous.
We here use the method of moving planes in integral form
which was introduced in \cite{CLOO,CLO}.

For a given real number $\lambda$, define
$$
\Sigma_{\lambda}=\left\lbrace  x=(x_1,x_2,x_3)|x_1>\lambda \right\rbrace.
$$
Write $x^{\lambda}=(2\lambda-x_1,x_2,x_3)$ and $
u_{\lambda}(x)=u(x^{\lambda})$.

First, we prove the following lemma.
\begin{lemma}\label{lem}
    Assume that $u$ solves (\ref{A1}). Then
    \begin{equation}\label{lem1}
        u(x)-u_{\lambda}(x)
        =\int_{\Sigma_{\lambda}}[\mathcal{K}_a(x-y)
        -\mathcal{K}_a(x^{\lambda}-y)][u^q(y)-u_{\lambda}^q(y)]dy.
    \end{equation}
\end{lemma}

\begin{proof}
In view of $|x-y^{\lambda}|=|x^{\lambda}-y|$,
there holds
    $$
    \begin{aligned}
u(x)=&\int_{\Sigma_{\lambda}}\mathcal{K}_a(x-y)u^q(y)dy
+\int_{\Sigma_{\lambda}^c}\mathcal{K}_a(x-y)u^q(y)dy\\
=&\int_{\Sigma_{\lambda}}\mathcal{K}_a(x-y)u^q(y)dy
+\int_{\Sigma_{\lambda}}\mathcal{K}_a(x-y^{\lambda})u^q(y^{\lambda})dy\\
=&\int_{\Sigma_{\lambda}}\mathcal{K}_a(x-y)u^q(y)dy
+\int_{\Sigma_{\lambda}}\mathcal{K}_a(x^{\lambda}-y)u_{\lambda}^q(y)dy.
    \end{aligned}
    $$
    Replacing $x^{\lambda}$ with $x$, we get
    $$
u_{\lambda}(x)=\int_{\Sigma_{\lambda}}\mathcal{K}_a(x^{\lambda}-y)u^q(y)dy
+\int_{\Sigma_{\lambda}}\mathcal{K}_a(x-y)u_{\lambda}^q(y)dy.
    $$
    Therefore,
    $$
u(x)-u_{\lambda}(x)
=\int_{\Sigma_{\lambda}}[\mathcal{K}_a(x-y)
-\mathcal{K}_a(x^{\lambda}-y)][u^q(y)-u_{\lambda}^q(y)]dy.
  $$
Lemma \ref{lem} is proved.
\end{proof}

Next, we prove the (ii) in Theorem \ref{th11}.

\begin{proof}
Write $\Sigma_{\lambda}^-=\left\lbrace  x \in
\Sigma_{\lambda}; u(x)<u_{\lambda}(x) \right\rbrace$.

{\it{Step 1.}} We claim that for sufficiently small $\lambda$,
$\Sigma_{\lambda}^-$ is an empty set.

According to Lemma {\ref{lem}} and the mean value theorem,
we deduce that for any $x\in \Sigma_{\lambda}^-$,
$$
u_{\lambda}(x)-u(x) \leq
\int_{\Sigma_{\lambda}^-}\mathcal{K}_a(x-y)[u_{\lambda}^q(y)-u^q(y)]dy
\leq  q
\int_{\Sigma_{\lambda}^-}\mathcal{K}_a(x-y)[u_{\lambda}^{q-1}(u_{\lambda}-u)](y)dy.
$$
Then, applying the classical Hardy-Littlewood-Sobolev inequality and the
H\"{o}lder inequality, we obtain
\begin{equation}\label{q1}
    \begin{aligned}
&\left \|  u_{\lambda}(x)-u(x) \right \|_{L^s(\Sigma_{\lambda}^-)}\\
\leq & q\left \|  \mathcal{K}_a\ast
[u_{\lambda}^{q-1}(u_{\lambda}-u)] \right
\|_{L^s(\Sigma_{\lambda}^-)}
\leq  C\left \|  I_2 [u_{\lambda}^{q-1}(u_{\lambda}-u)] \right
\|_{L^s(\Sigma_{\lambda}^-)}\\
\leq & C\left \| u_{\lambda}^{q-1}(u_{\lambda}-u) \right
\|_{L^{\frac{3s}{3+2s}}(\Sigma_{\lambda}^-)}
\leq  C\left \| u_{\lambda} \right \|_{L^{\frac{3(q-1)}{2}}
(\Sigma_{\lambda}^-)}^{q-1}
\left \|u_{\lambda}-u \right \|_{L^s(\Sigma_{\lambda}^-)} \\
\leq & C\left \| u_{\lambda} \right
\|_{L^{\frac{3(q-1)}{2}}(\Sigma_{\lambda})}^{q-1} \left
\|u_{\lambda}-u \right \|_{L^s(\Sigma_{\lambda}^-)} =  C\left \| u
\right \|_{L^{\frac{3(q-1)}{2}}(\Sigma_{\lambda}^c)}^{q-1}
    \left \|u_{\lambda}-u \right \|_{L^s(\Sigma_{\lambda}^-)} .
    \end{aligned}
\end{equation}

By virtue of $u \in L^{\frac{3(q-1)}{2}}(\mathbb{R}^3)$, we can choose
$R>0$ suitably large such that when $\lambda<-R$,
$$
C\left \| u \right
\|_{L^{\frac{3(q-1)}{2}}(\Sigma_{\lambda}^c)}^{q-1} \leq
\frac{1}{2}.
$$
Combining with \eqref{q1}, we see that
\begin{equation}\label{equ0}
    \left \|u_{\lambda}-u \right \|_{L^s(\Sigma_{\lambda}^-)}=0.
\end{equation}
Then $\Sigma_{\lambda}^-$ is an empty set. Namely,
\begin{equation}\label{qq}
u(x) \geq u_{\lambda}(x), \quad \forall x \in \Sigma_{\lambda}.
\end{equation}

{\it{Step 2.}} According to Step 1, we start move the plane
$T_{\lambda}=\left\lbrace x \in \mathbb{R}^3; x_1=\lambda \right\rbrace$
from the negative infinity of the $x_1$-direction to the right as
long as (\ref{qq}) holds.

Write
$$
\lambda_0=\sup \left\lbrace \lambda ;u(x) \geq u_{\mu}(x),
\quad \forall x \in \Sigma_{\mu}, \quad \mu \leq \lambda \right\rbrace.
$$
If $\lambda_0<0$, we claim that $u$ must be symmetric about the
plane $T_{\lambda_0}$, that is,
$$
u \equiv u_{\lambda_0}, \quad \forall x \in \Sigma_{\lambda_0}.
$$
Otherwise, we suppose that
$$
u(x) \geq u_{\lambda_0}(x) \quad and \quad
u(x) \not\equiv u_{\lambda_0}(x),
\quad x \in \Sigma_{\lambda_0}.
$$
We claim that the plane can be moved further to the right.
More precisely, there exists $\varepsilon >0$ such that
\begin{equation}\label{10}
    u(x) \geq u_\lambda(x), \quad \forall  \; x \in \varSigma_\lambda,
    \;\;\forall  \; \lambda \in [\lambda_0,\lambda_0+\varepsilon).
\end{equation}
This would contradict with the definition of $\lambda_0$.

In the case of $u(x) \neq u_{\lambda_0}(x)$ on
$\Sigma_{\lambda_0}$. According to (\ref{lem1}), we see that
$u(x)>u_{\lambda_0}(x)$ in the interior of $\Sigma_{\lambda_0}$.
Set
$$
\overline{\Sigma_{\lambda_0}^-}= \left\lbrace x\in
\Sigma_{\lambda_0} | u(x)\leq u_{\lambda_0}(x)\right\rbrace.
$$
Clearly, $\overline{\Sigma_{\lambda_0}^-}$ is an empty set and
$\lim_{\lambda \to \lambda_0}\Sigma_{\lambda}^- \subset
\overline{\Sigma_{\lambda_0}^-}$. Let $(\Sigma_{\lambda}^-)^*$ be
the reflection of a set  $\Sigma_{\lambda}^-$ about the plane
$x_1=\lambda$.

   By the same derivation of (\ref{q1}), there also holds
   $$
\left\|u_\lambda -u\right\|_{L^s(\Sigma_{\lambda}^-)} \leq
C\left\| u
\right\|_{L^{\frac{3(q-1)}{2}}(\Sigma_{\lambda}^-)^*}^{q-1}
\left\| u_\lambda -u \right\|_{L^s(\Sigma_{\lambda}^-)}.
   $$
The integrability condition $u \in L^{\frac{3(q-1)}{2}}(\mathbb{R}^3)$
guarantees that one can choose $\varepsilon$ sufficiently small
such that for all $\lambda \in [\lambda_0,\lambda_0+\varepsilon)$,
   $$
C\left\|
u\right\|_{L^{\frac{3(q-1)}{2}}(\Sigma_{\lambda}^-)^*}^{q-1} \leq
\frac{1}{2}.
   $$
In view of the above result, we see $\left\|u_\lambda
-u\right\|_{L^s(\Sigma_{\lambda}^-)} =0$. Hence, similar to the
derivation of (\ref{equ0}), we have that $\Sigma_{\lambda}^-$ must
be measure zero. Therefore, \eqref{10} holds and the contradiction
appears.

During the process of moving the plane from left to right along
$x_1$-direction, the plane either stops at some $\lambda<0$
or can be moved until $\lambda=0$. In the former case, we have
$u_{\lambda}(x)=u(x)$ for all $x\in \Sigma_{\lambda}$ by the
above argument. In the latter case, we obtain $u_{\lambda}(x)
\leq u(x)$ for all $x\in \Sigma_0$. At this point, we move the
plane from right to left along $x_1$-direction and notice that
$u_{\lambda}(x) \geq u(x)$ for all $x\in \Sigma_0$. Namely,
$u_{\lambda}(x)=u(x)$ for all $x\in \Sigma_0$. Since
$x_1$-direction can be chosen arbitrarily, we deduce that $u$ must
be radially symmetric and monotone decreasing about some point
$x_0 \in \mathbb{R}^3$. This completes the proof of (ii) in Theorem
\ref{th11}.
\end{proof}

\subsection{Decay rates}

In view of Theorem \ref{th2.1}, we are interested in the decay
rate of $u$ when $|x| \to \infty$.
In this subsection, we prove the conclusion (iii) in Theorem
\ref{th11}.

\begin{proof}
First we consider the lower bound estimation of $u$. In fact, for
sufficiently large $|x|$, it follows $|x|/2
\leq |x|-|y|\leq |x-y| \leq |x|+|y| \leq 2|x|$ when $y\in
B_{\delta}(0)$ with small $\delta \in (0,1/2)$.
In addition, $\mathcal{K}_a$ is decreasing.
Therefore, $\mathcal{K}_a(x-y) \geq
C\mathcal{K}_a(x)$ for some constant $C>0$. Thus,
    $$
u(x) \geq
\int_{B_{\delta}(0)}\frac{1-e^{-\frac{|x-y|}{a}}}{|x-y|}u^q(y)dy
\geq C\mathcal{K}_a(x)\int_{B_{\delta}(0)}u^q(y)dy \geq
C\mathcal{K}_a(x).
        $$

Next, we focus on the upper bound estimation of $u$.

    For fixed $R>0$, write
    $$
    L_1=\int_{B_R }\mathcal{K}_a(x-y)u^q(y)dy.
    $$
    Clearly,
    $$
\lim_{|x|\to
\infty}\left|\frac{\mathcal{K}_a(x-y)}{\mathcal{K}_a(x)}-1\right|=0
    $$
for all $y \in B_R$. This implies that for sufficiently large
$|x|$,
    \begin{equation}\label{L1}
        L_1 \leq 2\mathcal{K}_a(x)\int_{B_R }u^q(y)dy.
    \end{equation}

    Next, we write
    $$
L_2=\int_{(\mathbb{R}^3 \setminus {B_R})\setminus
B_{(1-2^{-m})|x|}(x)}\mathcal{K}_a(x-y)u^q(y)dy.
    $$
Here $m=m(x)>1$ will be determined later.
Obviously, when $y \in (\mathbb{R}^3 \setminus
{B_R})\setminus B_{(1-2^{-m})|x|}(x)$, we see that $|x-y|\geq
(1-2^{-m})|x|$. Since $\mathcal{K}_a$ is decreasing, there
holds
    \begin{equation}\label{L2}
L_2 \leq \mathcal{K}_a((1-2^{-m})x)\int_{{\mathbb{R}^3 \setminus
{B_R}}}u^q(y)dy.
    \end{equation}

    Finally, we write
    $$
    L_3=\int_{B_{(1-2^{-m})|x|}(x)}\mathcal{K}_a(x-y)u^q(y)dy.
    $$
According to (ii) in Theorem (\ref{th11}), we know that $u$ is
radially symmetric and decreasing about $x_0 \in \mathbb{R}^3$. Therefore,
for large $|x|$, wan can assume $x_0=0$. Of course, the
translation invariance of \eqref{A1} shows that there is a radial
solution centered at the origin in this class of radial solutions.
Thus,
    \begin{equation}\label{L3}
L_3 \leq
u^q(\frac{x}{2^m})\int_{B_{(1-2^{-m})|x|}(x)}\mathcal{K}_a(x-y)dy
\leq  u^q(\frac{x}{2^m})\int_{B_{|x|}(x)}{\frac{dy}{|x-y|}} \leq
Cu^q(\frac{x}{2^m})|x|^2,
            \end{equation}
where $C>0$ is independent of $m$.

Since \eqref{q3} implies $u \in L^s(\mathbb{R}^3)$ for $s \in (3,q)$, we deduce that
    $$
u^s(\frac{x}{2^m})\frac{|x|^3}{8^m} \leq C\int_{B_{|x|/2^{m}}(0)}u^s(y)dy
\leq C\|u\|_{L^s(\mathbb{R}^3)}^s.
    $$
    This implies
    \begin{equation}\label{L33}
        u^q(\frac{x}{2^m}) \leq C8^{qm/s}|x|^{-3q/s},
    \end{equation}
where $C>0$ is independent of $m$. Inserting (\ref{L33}) into
(\ref{L3}) yields
$$
L_3 \leq c_\ast 8^{qm/s}|x|^{2-\frac{3q}{s}},
$$
where the positive constant $c_*$ is independent of $m$. Let
$$
m(x)=1+\frac{s}{q\log 8}\log [c_*^{-1}\mathcal{K}_a(x)|x|^{3q/s-2}].
$$
Therefore, when $|x|$ is sufficiently large,
    \begin{equation}\label{L333}
        L_3 \leq 8^{q/s}\mathcal{K}_a(x),
    \end{equation}
and in view of $s \in (3,q)$, there holds
\begin{equation}\label{liu}
\lim_{|x| \to \infty}m(x) = \infty.
\end{equation}

Combining (\ref{L1}), (\ref{L2}) and (\ref{L333}), we get
    $$\begin{aligned}
        u(x) = &L_1+L_2+L_3\\
        \leq &\left(2\int_{B_R }u^q(y)dy+8^{q/s}\right)\mathcal{K}_a(x)
+\mathcal{K}_a((1-2^{-m})x)\int_{{\mathbb{R}^3 \setminus {B_R}}}u^q(y)dy
    \end{aligned}
    $$
for sufficiently large $|x|$. Therefore, by (\ref{liu}) it follows
$$
\lim_{|x| \to \infty}\frac{u(x)}{\mathcal{K}_a(x)}
\leq (2\|u\|_{L^q(\mathbb{R}^3)}^q+8^{q/s})+\|u\|_{L^q(\mathbb{R}^3)}^q
\lim_{|x| \to \infty}\frac{\mathcal{K}_a((1-2^{-m})x)}{\mathcal{K}_a(x)}
\leq C.
$$
This implies
    $$
    u(x) \leq C\mathcal{K}_a(x)
    $$
    when $|x|$ is sufficiently large.
    Thus, we prove (iii) in Theorem \ref{th11}.
\end{proof}

\section{Proof of equivalence is complete}

In this section, we prove Theorem \ref{theq} (ii).

\begin{theorem}\label{thD}
        Assume that the positive function $u\in L^{3(q-1)/2}(\mathbb{R}^3)$ solves (\ref{A1}).
    Then $u \in \mathcal{D} $ solves (\ref{A2}) in distribution sense.
    \end{theorem}

    \begin{proof}
    By Theorem \ref{tha1}, we see $q>3$. In view of \eqref{q3},
    we know $u^q \in L^p(\mathbb{R}^3)$ for all $1 \leq p <3/2$.
    Therefore, we can use Lemma 3.3 in \cite{PS} to see that
    $u$ solves (\ref{A2}) in distribution sense,
    and the following distributional derivatives
   \begin{equation}\label{jiang}
   \nabla u=(\nabla \mathcal{K}_a)\ast u^q \quad and \quad
    \Delta u=(\Delta \mathcal{K}_a) \ast u^q \quad a.e. \; in \; \mathbb{R}^3.
   \end{equation}

        Next, we only need to prove $u\in \mathcal{D}$.

    Clearly, \eqref{q3} implies
    \begin{equation}\label{yang}
    u\in L^6(\mathbb{R}^3) \cap L^q(\mathbb{R}^3).
    \end{equation}

        Take $\delta \in (0,1)$ suitably small and $R>0$ suitably large.

        By (\ref{wang}) and \eqref{yang}, for small $|x|$ we obtain
        \begin{equation}\label{503}
            \int_{B_{\delta}}|\nabla \mathcal{K}_a(x-y)|u^q(y)dy
    \leq C\int_{B_{\delta}}u^q(y)dy\leq C.
        \end{equation}
        By virtue of Theorem \ref{th11} (iii),
    we have
    \begin{equation}\label{qin}
    u(x)\leq \mathcal{K}_a(x)\leq C|x|^{-1}, \quad when \ \ |x| \ \ is \ \ large.
    \end{equation}
        Using this result and (\ref{wang}), we get
        \begin{equation}\label{504}
            \int_{\mathbb{R}^3\setminus B_R}|\nabla \mathcal{K}_a(x-y)|u^q(y)dy
    \leq C\int_{R}^{\infty}r^{3-(q+2)}\frac{dr}{r} \leq C
        \end{equation}
        when $|x|$ is small. At the same time,
        \begin{equation}\label{505}
            \int_{B_R\setminus B_{\delta}}|\nabla \mathcal{K}_a(x-y)|u^q(y)dy\leq C.
        \end{equation}
        Combining (\ref{503}), \eqref{504} and (\ref{505}), from \eqref{jiang} we deduce that
        \begin{equation}\label{506}
            |\nabla u(x)| \leq \int_{\mathbb{R}^3}|\nabla\mathcal{K}_a(x-y)|u^q(y)dy\leq C,
    \quad when \; |x| \; is \; small.
        \end{equation}

    When $|x|$ is large, by \eqref{qin} and (\ref{chu}), there holds
    \begin{equation}\label{507}
            \begin{aligned}
                &\int_{B_{\frac{|x|}{2}}(x)}|\nabla\mathcal{K}_a(x-y)|u^q(y)dy\\
            \leq & \frac{C}{|x|^q}\int_{B_{\frac{|x|}{2}}(x)}[\frac{1}{|x-y|^2}
    (e^{-\frac{|x-y|}{a}}+1)+\frac{1}{a|x-y|}e^{-\frac{|x-y|}{a}}]dy\\
             \leq & \frac{C}{|x|^q}\int_{0}^{\frac{|x|}{2}}
    [\frac{2}{r^2}+\frac{1}{ar}e^{-\frac{r}{a}}]\cdot r^3\frac{dr}{r}
                 \leq  \frac{C}{|x|^{q-1}}.
            \end{aligned}
    \end{equation}
    Obviously, when $y \in \mathbb{R}^3 \setminus B_{{|x|}/{2}}(x)$,
    there holds $|x-y|\geq \frac{|x|}{2}$. Thus, for large $|x|$, it follows
    \begin{equation}\label{508}
            \int_{\mathbb{R}^3\setminus B_{\frac{|x|}{2}}(x)}|\nabla\mathcal{K}_a(x-y)|u^q(y)dy\\
            \leq \frac{C}{|x|^2}\int_{\mathbb{R}^3\setminus B_{\frac{|x|}{2}}(x)}u^q(y)dy
            \leq \frac{C}{|x|^2}.
    \end{equation}
     Combining (\ref{507}) and (\ref{508}), from \eqref{jiang} we deduce that
     \begin{equation}\label{509}
        |\nabla u|\leq \int_{\mathbb{R}^3}|\nabla\mathcal{K}_a(x-y)|u^q(y)dy
     \leq \frac{C}{|x|^{q-1}}+ \frac{C}{|x|^2} \leq \frac{C}{|x|^2},
     \quad when \; |x| \; is \; large.
     \end{equation}

     When $x$ is double bounded, by \eqref{wang} and Theorem \ref{th11}, from
    \eqref{jiang} we also obtain $|\nabla u| \leq C$.
    Therefore, for $R>0$ suitably large, from (\ref{506}) and (\ref{509}) it follows
    $$
    \int_{\mathbb{R}^3}|\nabla u|^2dx \leq C+C\int_{\mathbb{R}^3\setminus B_R}\frac{1}{|x|^4}dx
    <\infty.
    $$
    Namely, $\nabla u \in L^2(\mathbb{R}^3)$, which and \eqref{yang} imply
    \begin{equation}\label{zhang}
    u \in D^{1,2}(\mathbb{R}^3).
    \end{equation}

    Finally, we prove that $\Delta u \in L^2(\mathbb{R}^3)$.

    According to Theorem \ref{th11} (i), $u$ is bounded.
    Therefore, for small $|x|$, by \eqref{feng} we have
    \begin{equation}\label{510}
            \int_{B_{\frac{|x|}{2}}(x)}|\Delta \mathcal{K}_a(x-y)|u^q(y)dy
            \leq  C\int_{B_{\frac{|x|}{2}}(x)}\frac{e^{-\frac{|x-y|}{a}}}{a^2|x-y|}dy
            \leq C\int_{0}^{\frac{|x|}{2}}\frac{e^{-\frac{r}{a}}}{r}\cdot r^3\frac{dr}{r}
            \leq C.
        \end{equation}
    Similarly, when $|x|$ is large, by \eqref{qin} we obtain that
    \begin{equation}\label{511}
        \int_{B_{\frac{|x|}{2}}(x)}|\Delta \mathcal{K}_a(x-y)|u^q(y)dy \leq \frac{C}{|x|^q}\int_{B_{\frac{|x|}{2}}(x)}\frac{e^{-\frac{|x-y|}{a}}}{a^2|x-y|}dy
    \leq \frac{C}{|x|^{q}}.
    \end{equation}
    In addition, $y\in \mathbb{R}^3\setminus B_{{|x|}/{2}}(x)$ implies $|x-y|\geq |x|/2$.
    By \eqref{feng}, there holds
    $|\Delta \mathcal{K}_a(x-y)| \leq 2e^{-\frac{|x|}{2a}}/(a^2|x|)$.
    This and \eqref{yang} imply that for all $x \in \mathbb{R}^3$,
    \begin{equation}\label{512}
        \int_{\mathbb{R}^3\setminus B_{\frac{|x|}{2}}(x)}|\Delta \mathcal{K}_a(x-y)|u^q(y)dy
    \leq \frac{C e^{-\frac{|x|}{2a}}}{|x|}.
    \end{equation}
    Combining (\ref{510})-(\ref{512}), from \eqref{jiang} we deduce that
    $$
        |\Delta u|\leq C+\frac{C e^{-\frac{|x|}{2a}}}{|x|}
    \leq \frac{C}{|x|}, \quad when \;|x|\; is \; small,
    $$
    and
    $$
        |\Delta u|\leq \frac{C}{|x|^q}+\frac{C e^{-\frac{|x|}{2a}}}{|x|}
    \leq \frac{C}{|x|^{q}}, \quad when \;|x|\; is \; large.
    $$
    Similarly, when $x$ is double bounded, by \eqref{feng} and Theorem
    \ref{th11}, from \eqref{jiang} we also deduce $|\Delta u| \leq C$.
    Combining these results, for suitably small $\theta \in (0,1)$
    and suitably large $R>1$, we get
    $$
    \int_{\mathbb{R}^3}|\Delta u|^2dx \leq C\int_{B_{\theta}}\frac{dx}{|x|^2}
    +C+C\int_{\mathbb{R}^3\setminus B_{R}}\frac{dx}{|x|^{2q}}
        <\infty.
    $$
    That is, $|\Delta u| \in L^2(\mathbb{R}^3)$. Combining with \eqref{zhang},
    we see $u\in \mathcal{D}$.
    This completes the proof of Theorem \ref{thD}.
    \end{proof}

\section{Liouville theorem of integral equation}

In this section, we prove conclusion (ii) in Theorem \ref{th1}.

\begin{theorem}
If $u \in L^{q+1}(\mathbb{R}^3)$ is a positive differentiable solution of
(\ref{A1}), then $q>5$.
\end{theorem}

\begin{proof}
According to Theorem \ref{tha1}, we have $q>3$.

{\it Step 1.} We prove several improper integrals are convergent.

In view of $u \in L^{q+1}(\mathbb{R}^3)$ we can find $R_j \to \infty$ ($j
\to \infty$) such that
\begin{equation}\label{ren}
  R_j\int_{\partial B_{R_j}}u^{q+1}ds \to 0.
\end{equation}

Integrating by parts we get
$$\begin{aligned}
\int_{B_R}[y \cdot \nabla u^{q}(y)]u(y)dy
&=\frac{q}{q+1}\int_{B_R}y \cdot \nabla u^{q+1}(y)dy\\
&=\frac{qR}{q+1}\int_{\partial B_R}u^{q+1}(y)ds
-\frac{3q}{q+1}\int_{B_R}u^{q+1}(y)dy
\end{aligned}
$$
Letting $R=R_j \to \infty$ and using \eqref{ren}, we obtain
\begin{equation}\label{gui}
  [y \cdot \nabla u^{q}(y)]u(y) \in L^1(\mathbb{R}^3),
\end{equation}
and
\begin{equation}\label{zi}
\begin{aligned}
  \int_{\mathbb{R}^3}[y \cdot \nabla u^{q}(y)]u(y)dy
  &=\frac{q}{q+1}\int_{\mathbb{R}^3}y \cdot \nabla u^{q+1}(y)dy\\
  &=-\frac{3q}{q+1}\int_{\mathbb{R}^3}u^{q+1}(y)dy.
 \end{aligned}
\end{equation}

On the other hand, by the H\"older inequality, there holds
$$
\int_{\mathbb{R}^3}u^q(y)e^{-\frac{|x-y|}{a}}dy
\leq \left(\int_{\mathbb{R}^3} u^{q+1}(y)dy\right)^{\frac{q}{q+1}}
\left(\int_{\mathbb{R}^3}e^{-\frac{(q+1)|x-y|}{a}}dy\right)^{\frac{1}{q+1}}.
$$
Thus,
\begin{equation}\label{mao}
u^q(\cdot)e^{-\frac{|x-\cdot|}{a}} \in L^1(\mathbb{R}^3), \quad for \ \
a.e. \ \ x \in \mathbb{R}^3.
\end{equation}

Next, we claim that
\begin{equation}\label{chou}
\int_{\mathbb{R}^3}[y \cdot \nabla u^q(y)] \mathcal{K}_a(x-y) dy < \infty,
\quad for \ \ a.e. \ \ x \in \mathbb{R}^3.
\end{equation}
In fact, integrating by parts we get
\begin{equation}\label{you}
\begin{aligned}
\int_{B_R}[y \cdot \nabla u^q(y)] \mathcal{K}_a(x-y)dy
=&R\int_{\partial B_R} u^q(y) \mathcal{K}_a(x-y)ds
-3\int_{B_R} u^q(y) \mathcal{K}_a(x-y)dy\\
&-\int_{B_R}u^q(y)[y \cdot \nabla \mathcal{K}_a(x-y)]dy.
\end{aligned}
\end{equation}
When $R \to \infty$, the defects of the improper integral in the
second term of the right hand side of \eqref{you} may happen at
$x$ or $\infty$. When $y$ is near $\infty$, by $q>3$ we have
$$
\int_{\mathbb{R}^3 \setminus B_{R+|x|}}u^q(y) \mathcal{K}_a(x-y)dy
\leq C \|u\|_{L^{q+1}(\mathbb{R}^3)}^q
\left(\int_{R+|x|}^\infty \frac{r^3}{r^{q+1}}
\frac{dr}{r}\right)^{\frac{1}{q+1}}
<\infty.
$$
When $y$ is near $x$, we can find $\delta \in (0,1/2)$, such that
$\mathcal{K}_a(x-y) \leq 2/a$ when $|x-y|<\delta$.
$$
\int_{B_\delta(x)}u^q(y) \mathcal{K}_a(x-y)dy
\leq C(\delta) \|u\|_{L^{q+1}(\mathbb{R}^3)}^q<\infty.
$$
Combining these two estimates we get
$$
u^q(\cdot) \mathcal{K}_a(x-\cdot) \in L^1(\mathbb{R}^3),
\quad for \ \ a.e. \ \ x \in \mathbb{R}^3.
$$
Therefore, we can find $R=R_j \to \infty$ such that the first term
of the right hand side of \eqref{you} converges to zero.

By \eqref{wang}, we have
\begin{equation}\label{hai}
\begin{cases}
|\nabla \mathcal{K}_a(x-y)| \leq C,
\quad when \ y \ is \ near \ x;\\[3mm]
|\nabla \mathcal{K}_a(x-y)| \leq C|x-y|^{-2},
\quad when \ y \ is \ near \ \infty.
\end{cases}
\end{equation}
Therefore,
$$
\int_{B_\delta(x)}u^q(y)|y \cdot \nabla \mathcal{K}_a(x-y)|dy
\leq C(\delta)\|u\|_{L^{q+1}(\mathbb{R}^3)}^q|x|<\infty.
$$
$$
\int_{\mathbb{R}^3 \setminus B_{R+|x|}}u^q(y)|y \cdot \nabla
\mathcal{K}_a(x-y)|dy \leq
C\|u\|_{L^{q+1}(\mathbb{R}^3)}^q\left(\int_{R+|x|}^\infty
\frac{r^3}{r^{q+1}}\frac{dr}{r}\right)^{\frac{1}{q+1}} <\infty.
$$
Combining these two estimates, we know that the third term of the
right hand side of \eqref{you} is convergent when $R=R_j \to
\infty$, \eqref{chou} is proved.

Insert \eqref{A1} into the left hand side of \eqref{zi}. In view
of \eqref{gui} and \eqref{chou} we obtain
\begin{equation}\label{yin}
\begin{aligned}
\int_{\mathbb{R}^3}[y \cdot \nabla u^{q}(y)]u(y)dy
&=\int_{\mathbb{R}^3}[y \cdot \nabla u^{q}(y)]
\int_{\mathbb{R}^3}\mathcal{K}_a(x-y)u^q(x)dxdy\\
&=\int_{\mathbb{R}^3}u^q(x)\int_{\mathbb{R}^3}[y \cdot \nabla
u^{q}(y)]\mathcal{K}_a(x-y)dy dx.
  \end{aligned}
\end{equation}

{\it Step 2.} We use the Pohozaev identity in integral form to prove theorem.

    From (\ref{A1}), it follows
    $$
u(\mu x) =
{\mu}^2\int_{\mathbb{R}^3}\frac{1-e^{-\frac{\mu|x-y|}{a}}}{|x-y|}u^q(\mu
y)dy
    $$
Differentiating both sides with respect to $\mu$ and then letting
$\mu=1$, we get
    $$\begin{aligned}
x \cdot \nabla u(x)
= &\left[\frac{\mathrm{d} u(\mu x)}{\mathrm{d} \mu } \right]_{\mu =1}\\
= &2u + \int_{\mathbb{R}^3} [y \cdot \nabla u^q(y)]
\mathcal{K}_a(x-y)dy+\frac{1}{a}\int_{\mathbb{R}^3}u^q(y)e^{-\frac{|x-y|}{a}}dy.
    \end{aligned}
    $$
In fact, it makes sense if we notice \eqref{mao} and \eqref{chou}.

Multiply the result above by $u^q$ and integrate on $B_R$. Letting
$R \to \infty$ and using \eqref{yin}, we obtain that
$$
\lim_{R \to \infty} \int_{B_R}u^q(x)\int_{\mathbb{R}^3}u^q(y)e^{-\frac{|x-y|}{a}}dy dx
$$
exists, and
    $$\begin{aligned}
\frac{1}{q+1}\int_{\mathbb{R}^3} x \cdot \nabla u^{q+1}(x)dx
=& 2\int_{\mathbb{R}^3} u^{q+1}(x)dx+\int_{\mathbb{R}^3}[y \cdot \nabla u^{q}(y)]u(y)dy\\
&+\frac{1}{a}\int_{\mathbb{R}^3}\int_{\mathbb{R}^3}u^q(x)u^q(y)e^{-\frac{|x-y|}{a}}dxdy.
    \end{aligned}
    $$
    Applying \eqref{zi}, we can see that
    $$
\left[\frac{3(q-1)}{q+1}-2\right]\int_{\mathbb{R}^3}u^{q+1}(x)
=\frac{1}{a}\int_{\mathbb{R}^3}\int_{\mathbb{R}^3}u^q(x)u^q(y)e^{-\frac{|x-y|}{a}}dxdy.
    $$
Since the right hand side is positive, we obtain that
$3(q-1)/(q+1)>2$, which implies $q>5$.
\end{proof}

\section{Allen-Cahn equation}

In this section, we prove  Theorems \ref{th77} and \ref{th78}.

\subsection{Proof of Theorem \ref{th77}}

First, we have the following theorem.

\begin{theorem}\label{th79}
Under the same assumption of Theorem \ref{th77}, then one of the
following results holds true

    (i) $u \in L^{q-1}(\mathbb{R}^3)$ and $\lim_{|x|\to \infty}u(x) =0$;

    (ii) $1-u^q \in L^1(\mathbb{R}^3)$ and $\lim_{|x| \to \infty}u(x)=1$.
\end{theorem}

\begin{proof}
The ideas in \cite{BMR} is used here.
We claim that $S_\ast:=\left\lbrace  x \in \mathbb{R}^3; \frac{1}{4} \leq u
\leq \frac{3}{4} \right\rbrace $ is bounded. Namely, there exists
suitably large $R_0>0$ such that $S_\ast \subset B_{R_0}(0)$.

Otherwise, we can find a sequence $\left\lbrace x_j\right\rbrace
\subset S_\ast$ satisfying $\lim_{j \to \infty}|x_j|=\infty$. Since
$u$ is uniformly continuous, there exists $\eta \in (0,1)$ such
that
    $$
\frac{1}{8} \leq |u(x)| \leq \frac{7}{8}, \quad for \;
|x-x_j|<\eta,\; \forall j.
    $$
Choose a subsequence of ${x_j}$ denoted by itself such that
$|x_i-x_j|>3\eta$ for $i \neq j$.
    Therefore,
    $$
\int_{\cup_j B(x_j,\eta)}u^{q-1}(1-u^q)dx \geq
C|\cup_jB(x_j,\eta)|=\infty,
    $$
    which contradicts with (\ref{81}).

Since $u$ is uniformly continuous and $\mathbb{R}^3\setminus B_{R_0}(0)$ is
connected, either $0\leq u \leq 1/4$ or $3/4 \leq u \leq 1$ holds
true on $\mathbb{R}^3\setminus B_{R_0}(0)$.

When $0\leq u \leq 1/4$ on $\mathbb{R}^3\setminus B_{R_0}(0)$, by
(\ref{80}) and (\ref{81}) we get
    $$
    \begin{aligned}
&\int_{\mathbb{R}^3}u^{q-1}dx=\int_{B_{R_0}(0)}u^{q-1}dx
+\int_{\mathbb{R}^3 \setminus B_{R_0}(0)}u^{q-1}dx\\
&\leq |B_{R_0}|+C\int_{\mathbb{R}^3 \setminus
B_{R_0}(0)}u^{q-1}(1-u^q)dx<\infty,
    \end{aligned}
    $$
    Namely, $u\in L^{q-1}(\mathbb{R}^3)$. Now, we claim that
    \begin{equation}
        \lim_{|x| \to \infty} u(x)=0.
    \end{equation}

Otherwise, we can find $\varepsilon_0>0$ and $|x_j| \to \infty$ such
that $u(x_j) \geq 2\varepsilon_0$. Since $u$ is uniformly
continuous, there exists $\eta >0$ such that
$|u(x)-u(y)|<\varepsilon_0$ when $|x-y|<\eta$, which leads to $u(x)
> u(x_j)-\varepsilon_0 \geq \varepsilon_0$ for $x \in
B_{\eta}(x_j)$. Therefore, when $|x_j| \to \infty$,
    $$
{\varepsilon_0}^{q-1}|B_{\eta}|<\int_{B_{\eta}(x_j)}u^{q-1}(x)dx
\to 0 \quad \quad (by \ \ u\in L^{q-1}(\mathbb{R}^3)).
    $$
    This is impossible.

When $3/4\leq u \leq 1$ on $\mathbb{R}^3\setminus B_{R_0}(0)$, by the
same argument above, from
(\ref{80}) and (\ref{81}) we can deduce $1-u^q \in L^1(\mathbb{R}^3)$, and
    \begin{equation*}
        \lim_{|x| \to \infty} u(x)=1.
    \end{equation*}
    Thus, the proof of Theorem \ref{th79} is complete.
\end{proof}

\textbf{Proof of Theorem \ref{th77}.}

\begin{proof}
    Let
    \begin{equation}\label{811}
v(x)=\int_{\mathbb{R}^3}\frac{(1-e^{-\frac{|x-y|}{a}})
u^{q-1}(y)(1-u^q(y))}{|x-y|}dy
    \end{equation}
    We claim that
    \begin{equation}\label{822}
        \lim_{|x| \to \infty} v(x)=0.
    \end{equation}

Take $x_0 \in \mathbb{R}^3$. In view of $v(x_0) \leq
I_2(u^{q-1}(1-u^q))(x_0)$, by the same derivation of \eqref{sun},
we can also deduce that
    $$
v(x_0) \leq \varepsilon+CI_2(u^{q-1}(1-u^q))(x), \quad when \;
|x-x_0|<\delta.
    $$
    For $\forall s>1$, we have
    \begin{equation}\label{833}
        \begin{aligned}
v^s(x_0)=&|B_{\delta}(x_0)|^{-1}\int_{B_{\delta}(x_0)}v^s(x_0)dx\\
\leq
&C{\varepsilon}^s+C\int_{B_{\delta}(x_0)}[I_2(u^{q-1}(1-u^q))]^s(x)dx.
        \end{aligned}
    \end{equation}

By the classical Hardy-Littlewood-Sobolev inequality, we get
    $$
\left \| I_2(u^{q-1}(1-u^q)) \right \|_{L^s(\mathbb{R}^3)} \leq C\left \|
u^{q-1}(1-u^q) \right \|_{L^{\frac{3s}{3+2s}}(\mathbb{R}^3)}
    $$
as long as $s>3$.
According to Theorem \ref{th79},
$u\in L^{q-1}(\mathbb{R}^3)$ or $1-u^q \in L^1(\mathbb{R}^3)$.
Noting (\ref{80}) we have
$u\in L^{t}(\mathbb{R}^3)$ for all $t \in [q-1,\infty]$,
or $1-u^q \in L^t(\mathbb{R}^3)$ for all $t \in [1,\infty]$.
Therefore, the result above shows that for some $s>3$,
    $$
       I_2(u^{q-1}(1-u^q)) \in L^s(\mathbb{R}^3).
    $$
Therefore, it follows that
    $$
\lim_{|x_0| \to
\infty}\int_{B_{\delta}(x_0)}[I_2(u^{q-1}(1-u^q))]^s(x)dx = 0.
    $$
Inserting this result into (\ref{833}), we can see (\ref{822}).

    From (\ref{A3}), it follows that
    $$
    \lim_{|x| \to \infty} u(x) = l.
    $$
According to Theorem \ref{th79}, we can see the conclusions of
Theorem \ref{th77}.
\end{proof}

\subsection{Proof of Theorem \ref{th78}}

\begin{proof}
    {\it Step 1.} Case of $1-u^q \in L^1(\mathbb{R}^3)$.

    According to Theorem \ref{th77}, $l=1$ and (\ref{A3}) becomes
    $$
    u(x)=1+C_\ast\int_{\mathbb{R}^3}\mathcal{K}_a(x-y)u^{q-1}(y)(1-u^q(y))dy.
    $$
    Combining this with (\ref{80}), for $C_*>0$, we deduce that
    $$
1 \geq u(x) =
1+C_\ast\int_{\mathbb{R}^3}\mathcal{K}_a(x-y)u^{q-1}(y)(1-u^q(y))dy \geq 1.
    $$
Namely, $u \equiv 1$ on $\mathbb{R}^3$.

    {\it Step 2.} Case of $u \in L^{q-1}(\mathbb{R}^3)$ and $q \in (1,6]$.

    According to Theorem \ref{th77}, $l=0$ and
    (\ref{A3}) becomes
    \begin{equation}\label{85}
        u(x)=C_\ast\int_{\mathbb{R}^3}\mathcal{K}_a(x-y)[u^{q-1}(1-u^q)](y)dy.
    \end{equation}

    {\it Substep 2.1.} we claim that the following two improper integrals are convergent
    for all $x \in \mathbb{R}^3$:
    \begin{equation}\label{zhou}
    \int_{\mathbb{R}^3}u^{q-1}(y)(1-u^q(y))e^{-\frac{|x-y|}{a}}dy<\infty,
    \end{equation}
    and
        \begin{equation}\label{zheng}
        \int_{\mathbb{R}^3}y\cdot \nabla[u^{q-1}(1-u^q)]\mathcal{K}_a(x-y)dy
        <\infty.
        \end{equation}
Clearly, $u \in L^{q-1}(\mathbb{R}^3)$ and \eqref{80} imply that \eqref{zhou} is true.

 \textit{Proof of \eqref{zheng}.}
        In fact, by \eqref{85} and \eqref{80}, there holds
                \begin{equation}\label{tao}
                \int_{\mathbb{R}^3}u^{q-1}(y)(1-u^q(y))\mathcal{K}_a(x-y)dy
        =C_\ast^{-1} u(x) \leq C_\ast^{-1}
                \end{equation}
                for all $x \in \mathbb{R}^3$.
        Therefore, we can find $R=R_j\to \infty$ such that
        \begin{equation}\label{702}
            R\int_{\partial B_R}u^{q-1}(y)(1-u^q(y))\mathcal{K}_a(x-y)ds
            \to 0.
        \end{equation}

        Next, we claim that the improper integrals
        \begin{equation}\label{703}
            H_1(\mathbb{R}^3):=\frac{1}{a}\int_{\mathbb{R}^3}u^{q-1}(y)(1-u^q(y))
            \frac{e^{-\frac{|x-y|}{a}}(x-y)\cdot y}{|x-y|^2}dy
        \end{equation}
        and
        \begin{equation}\label{705}
            H_2(\mathbb{R}^3):=\int_{\mathbb{R}^3}u^{q-1}(y)(1-u^q(y))
            \frac{1-e^{-\frac{|x-y|}{a}}}{|x-y|^3}(x-y)\cdot ydy
        \end{equation}
        absolutely converge for each $x\in \mathbb{R}^3$.

        In fact, we notice that the defect points of both
        $H_1(\mathbb{R}^3)$ and $H_2(\mathbb{R}^3)$ are $x$ and $\infty$.

        When $y$ is near $\infty$, we obtain that, by (\ref{81}),
        \begin{equation}\label{704}
        |H_1(\mathbb{R}^3\setminus B_R)|
        \leq C\int_{\mathbb{R}^3 \setminus B_R}u^{q-1}(y)(1-u^q(y))dy
        <\infty,
        \end{equation}
        and
        \begin{equation}\label{706}
            |H_2(\mathbb{R}^3\setminus B_R)|
            \leq C\int_{\mathbb{R}^3 \setminus B_R}u^{q-1}(y)(1-u^q(y))
            \frac{1-e^{-\frac{|x-y|}{a}}}{|x-y|}dy\leq Cu(x)
            <\infty.
        \end{equation}
         When $y$ is near $x$, we obtain that, by (\ref{80}),
         $$
         |H_1(B_{\delta}(x))|+|H_2(B_{\delta}(x))|
         \leq C\int_{B_{\delta}(x)}u^{q-1}(y)(1-u^q(y))\frac{dy}{|x-y|}
         <\infty.
         $$
         Here $\delta \in (0,1/2)$ is suitably small.
         Combining this result with (\ref{704}) and (\ref{706}),
         we prove that (\ref{703}) and (\ref{705}) are absolutely convergent.

        Now, we can verify (\ref{zheng}). In fact, integrating by parts yields
        \begin{equation}\label{707}
            \begin{aligned}
                &\int_{B_R}y\cdot \nabla[u^{q-1}(y)(1-u^q(y))]\mathcal{K}_a(x-y)dy\\
                =&R\int_{\partial B_R}u^{q-1}(y)(1-u^q(y))\mathcal{K}_a(x-y)ds\\
                &-3\int_{B_R}u^{q-1}(y)(1-u^q(y))\mathcal{K}_a(x-y)dy\\
                &-\frac{1}{a}\int_{B_R}u^{q-1}(y)(1-u^q(y))
                \frac{e^{-\frac{|x-y|}{a}}(x-y)\cdot y}{|x-y|^2}dy\\
                &+\int_{B_R}u^{q-1}(y)(1-u^q(y))
                \frac{1-e^{-\frac{|x-y|}{a}}}{|x-y|^3}(x-y)\cdot ydy.
            \end{aligned}
        \end{equation}
        Letting $R=R_j\to \infty$ in (\ref{707}) and using
        (\ref{tao}) and (\ref{702}), we derive that
        $$
        \int_{\mathbb{R}^3}y\cdot \nabla[u^{q-1}(y)(1-u^q(y))]\mathcal{K}_a(x-y)dy
        =-\frac{3}{C_*}u(x)-H_1(\mathbb{R}^3)+H_2(\mathbb{R}^3),
        $$
        and hence it is convergent at each $x\in \mathbb{R}^3$
        because \eqref{703} and \eqref{705} are absolutely convergent.
    The proof of \eqref{zheng} is complete.

    {\it Substep 2.2.} For any $\mu >0$, from \eqref{85} it follows
    $$
u(\mu
x)=C_\ast{\mu}^2\int_{\mathbb{R}^3}\frac{1-e^{-\frac{\mu|x-y|}{a}}}
{|x-y|}[u^{q-1}(1-u^q)](\mu
y)dy.
    $$
    Thus,
    \begin{equation}\label{99}
        \begin{aligned}
x \cdot \nabla u(x)=\left[\frac{\mathrm{d} u(\mu x)}{\mathrm{d}
\mu}\right]_{\mu =1}
=&2u+C_\ast\int_{\mathbb{R}^3}y\cdot \nabla[u^{q-1}(1-u^q)]\mathcal{K}_a(x-y)dy\\
&+\frac{C_\ast}{a}\int_{\mathbb{R}^3}u^{q-1}(y)(1-u^q(y))e^{-\frac{|x-y|}{a}}dy.
        \end{aligned}
    \end{equation}
In view of \eqref{zhou} and \eqref{zheng}, the right hand side of \eqref{99}
makes sense.

In view of \eqref{80} and $u\in L^{q-1}(\mathbb{R}^3)$, we obtain $u\in
L^{2q}(\mathbb{R}^3)$. Therefore, we can find $R=R_j\to \infty$ such that
    \begin{equation}\label{100}
R\int_{\partial B_{R}} u^q(x)ds + R\int_{\partial B_{R}}
u^{2q}(x)ds \to 0.
    \end{equation}

    Integrating by parts, we have
    \begin{equation}\label{102}
        \begin{aligned}
\int_{B_R}u^{q-1}(x)(x \cdot \nabla
u(x))dx=&\frac{1}{q}\int_{B_R}x \cdot \nabla
u^q(x)dx\\
=&\frac{R}{q}\int_{\partial{B_R}}u^q(x)ds
-\frac{3}{q}\int_{B_R}u^q(x)dx,
        \end{aligned}
    \end{equation}
    \begin{equation}\label{103}
        \begin{aligned}
\int_{B_R}u^{2q-1}(x)(x \cdot \nabla u(x))dx
=&\frac{1}{2q}\int_{B_R}x \cdot \nabla u^{2q}(x)dx\\
=&\frac{R}{2q}\int_{\partial{B_R}}u^{2q}(x)ds
-\frac{3}{2q}\int_{B_R}u^{2q}(x)dx.
        \end{aligned}
    \end{equation}
    Combining with (\ref{100}), we see that
    \begin{equation}\label{104}
        \int_{\mathbb{R}^3}u^{q-1}(x)(x \cdot \nabla u(x))dx
        =-\frac{3}{q}\int_{\mathbb{R}^3}u^q(x)dx
    \end{equation}
    and
    \begin{equation}\label{105}
        \int_{\mathbb{R}^3}u^{2q-1}(x)(x \cdot \nabla u(x))dx
        =-\frac{3}{2q}\int_{\mathbb{R}^3}u^{2q}(x)dx.
    \end{equation}
    These results show that
    \begin{equation}\label{106}
\int_{\mathbb{R}^3}u^{q-1}(x)(1-u^q(x))(x \cdot \nabla u(x))dx<\infty.
    \end{equation}

    In addition, (\ref{80}) and $u\in L^{q-1}(\mathbb{R}^3)$ also lead to
    \begin{equation}\label{114}
        \int_{\mathbb{R}^3}u^q(x)(1-u^q(x))dx<\infty
    \end{equation}
    and
    \begin{equation}\label{1066}
\int_{\mathbb{R}^3}u^{q-1}(x)(1-u^q(x))\int_{\mathbb{R}^3}u^{q-1}(y)(1-u^q(y))
e^{-\frac{|x-y|}{a}}dydx<\infty.
    \end{equation}

Multiply (\ref{99}) by $u^{q-1}(x)(1-u^q(x))$ and integrate over
$B_R$. Letting $R \to \infty$, we have
    \begin{equation}\label{107}
        \begin{aligned}
&\int_{\mathbb{R}^3}u^{q-1}(x)(1-u^q(x))(x \cdot \nabla u(x))dx
-2\int_{\mathbb{R}^3}u^q(x)(1-u^q(x))dx\\
=&C_\ast \lim_{R \to \infty}\int_{B_R}u^{q-1}(x)(1-u^q(x))
\int_{\mathbb{R}^3}y\cdot \nabla[u^{q-1}(1-u^q)]\mathcal{K}_a(x-y)dydx\\
&+\frac{C_\ast}{a}\int_{\mathbb{R}^3}u^{q-1}(x)(1-u^q(x))
\int_{\mathbb{R}^3}u^{q-1}(y)(1-u^q)(y)e^{-\frac{|x-y|}{a}}dydx.
        \end{aligned}
    \end{equation}
By \eqref{106}, \eqref{114} and \eqref{1066}, from \eqref{107} it
follows that the first term of the right hand side exists and is
equal to
    \begin{equation}\label{110}
        C_\ast \int_{\mathbb{R}^3}u^{q-1}(x)(1-u^q(x))\int_{\mathbb{R}^3}y\cdot
        \nabla [u^{q-1}(1-u^q)]\mathcal{K}_a(x-y)dydx<\infty.
    \end{equation}
By using the Fubini theorem and (\ref{85}), we obtain
    \begin{equation}\label{108}
        \begin{aligned}
&C_\ast\int_{\mathbb{R}^3}u^{q-1}(x)(1-u^q(x))\int_{\mathbb{R}^3}y\cdot
\nabla [u^{q-1}(1-u^q)]\mathcal{K}_a(x-y)dydx\\
=&C_\ast\int_{\mathbb{R}^3}y\cdot \nabla[u^{q-1}(1-u^q)]
\int_{\mathbb{R}^3}u^{q-1}(x)(1-u^q(x))\mathcal{K}_a(x-y)dxdy\\
=&\int_{\mathbb{R}^3}(x \cdot \nabla[u^{q-1}(1-u^q)])u(x)dx\\
=&(q-1)\int_{\mathbb{R}^3}u^{q-1}(x \cdot \nabla u(x))dx
-(2q-1)\int_{\mathbb{R}^3}u^{2q-1}(x \cdot \nabla u(x))dx.
        \end{aligned}
    \end{equation}

Inserting this result into (\ref{107}), and combining (\ref{104})
and (\ref{105}), we get
    \begin{equation}\label{109}
        \begin{aligned}
&\frac{q-6}{q}\int_{\mathbb{R}^3}u^q(x)dx+\frac{3-q}{q}\int_{\mathbb{R}^3}u^{2q}(x)dx\\
=&\frac{C_\ast}{a}\int_{\mathbb{R}^3}u^{q-1}(x)(1-u^q(x))
\int_{\mathbb{R}^3}u^{q-1}(y)(1-u^q(y))e^{-\frac{|x-y|}{a}}dydx
        \end{aligned}
    \end{equation}
When $3 \leq q \leq 6$, the left hand side of (\ref{109}) is
non-positive. In addition, the left hand side of (\ref{109}) is equal to
    $$
-\frac{3}{q}\int_{\mathbb{R}^3}u^q(x)dx
+\frac{q-3}{q}\left(\int_{\mathbb{R}^3}u^q(x)dx
-\int_{\mathbb{R}^3}u^{2q}(x)dx\right).
    $$
In view of $\|u\|_{L^q(\mathbb{R}^3)}^q \geq \|u\|_{L^{2q}(\mathbb{R}^3)}^{2q}$,
the left hand side of (\ref{109}) is also non-positive when
$1<q<3$. However, the right hand side of (\ref{109}) is non-negative.
Therefore, we can easily deduce $u \equiv 0$ when $1<q \leq 6$.

    Thus, we complete the proof of the Theorem
(\ref{th78}).
\end{proof}

\paragraph{Acknowledgements.} This research was supported
by the Natural Science Foundation of Jiangsu (No. BK20241878).

\paragraph{Data availability.} Not applicable to this article as no datasets were generated or analysed during the current
study.

\paragraph{Disclosure statement.} This work does not have any conflicts of interest.

{\sc Tiantian Zhou}

Institute of Mathematics, School of Mathematical Sciences

Nanjing Normal University, Nanjing, 210023, China

Email:ztt0515@foxmail.com

\vskip 5mm

{\sc Yutian Lei}

Ministry of Education Key Laboratory for NSLSCS, School of Mathematical Sciences

Nanjing Normal University, Nanjing, 210023, China

Email: leiyutian@njnu.edu.cn

\end{document}